\documentclass[12pt,reqno]{article}
\usepackage{geometry}                
\usepackage{graphicx}
\usepackage{amsmath, amsthm,amsfonts,amssymb}
\usepackage{epstopdf}
\usepackage{authblk}
\usepackage[letterpaper=true,colorlinks=true,urlcolor=blue,linkcolor=blue,citecolor=blue,pdfstartview=FitH]{hyperref}

\allowdisplaybreaks[1]
\usepackage[sort,comma]{natbib}
\newtheorem{thm}{Theorem}[section]
\newtheorem{corollary}[thm]{Corollary}

\newtheorem{prop}[thm]{Proposition}
\newtheorem{lem}[thm]{Lemma}
\theoremstyle{definition}
\newtheorem{defn}[thm]{Definition}

\newtheorem{cnd}[thm]{Condition}

\newcommand{\thmref}[1]{Theorem~{\rm \ref{#1}}}
\newcommand{\lemref}[1]{Lemma~{\rm \ref{#1}}}
\newcommand{\corref}[1]{Corollary~{\rm \ref{#1}}}
\newcommand{\cndref}[1]{Condition~{\rm \ref{#1}}}
\newcommand{\propref}[1]{Proposition~{\rm \ref{#1}}}
\newcommand{\defref}[1]{Definition~{\rm \ref{#1}}}
\newcommand{\remref}[1]{Remark~{\rm \ref{#1}}}

\newcommand{\sectref}[1]{Section~{\rm \ref{#1}}}

\newtheorem{rem}[thm]{Remark}

\def\R{\ensuremath {\mathbb R}}

\newcommand{\e}{\varepsilon}

\renewcommand{\P}{\ensuremath{\mathbb P}}

\newcommand{\E}{\ensuremath{\mathbb E}}

\newcommand{\I}{\ensuremath{\mathcal I}}
\newcommand{\N}{\ensuremath{\mathbb N}}

\newcommand{\A}{\ensuremath{\mathcal A}}
\newcommand{\wdt}{\widetilde}
\renewcommand{\d}{\mathrm{\, d}}

\newcommand{\set}[1]{\left\{#1\right\}}

\newcommand{\lbar}{\overline}
\renewcommand{\(}{\left(}
\renewcommand{\)}{\right)}

\makeatletter \@addtoreset{equation}{section}

\allowdisplaybreaks

\title{Continuous Inventory Models of Diffusion Type: Long-term Average Cost Criterion\thanks{This research was supported in part by     the Simons Foundation under grant award 246271 and a grant from UWM Research Growth Initiative.}}
\author[1]{K.L. Helmes}
\author[2]{R.H. Stockbridge}
\author[2]{C. Zhu}
\affil[1]{\small Institute for Operations Research, Humboldt University of Berlin, Spandauer Str. 1, 10178, Berlin, Germany, {\tt helmes@wiwi.hu-berlin.de}}
\affil[2]{Department of Mathematical Sciences,   University of Wisconsin-Milwaukee,   Milwaukee, WI 53201,   USA,   {\tt stockbri@uwm.edu}, {\tt zhu@uwm.edu}}

\begin{document}

\maketitle

\abstract{This paper establishes conditions for optimality of an $(s,S)$ ordering policy for the minimization of the long-term average cost of one-dimensional diffusion inventory models.  The class of such models under consideration have general drift and diffusion coefficients and boundary points that are consistent with the notion that demand should tend to decrease the inventory level.  Characterization of the cost of a general $(s,S)$ policy as a function $F$ of two variables naturally leads to a nonlinear optimization problem over the ordering levels $s$ and $S$.  Existence of an optimizing pair $(s_*,S_*)$ is established for these models.  Using the minimal value $F_*$ of $F$, along with $(s_*,S_*)$, a function $G$ is identified which is proven to be a solution of a quasi-variational inequality provided a simple condition holds.  At this level of generality, optimality of the $(s_*,S_*)$ ordering policy is established within a large class of ordering policies such that local martingale and transversality conditions involving $G$ hold.  For specific models, optimality of an $(s,S)$ policy in the general class of admissible policies can be established using comparison results.  This most general optimality result is shown for the classical drifted Brownian motion inventory model with holding and fixed plus proportional ordering costs and for a geometric Brownian motion inventory model with fixed plus level-dependent ordering costs.  However, for a drifted Brownian motion process with reflection at $\{0\}$, a new class of non-Markovian policies is introduced which have lower costs than the $(s,S)$ policies.  In addition, interpreting reflection at $\{0\}$ as ``just-in-time'' ordering, a necessary and sufficient condition is given that determines when just-in-time ordering is better than traditional $(s,S)$ policies.
}
\smallskip

\smallskip 


  \section{Introduction}
This paper examines the optimal ordering for single-item inventory processes in the presence of fixed and level-dependent order costs as well as holding and back-order costs.  The presence of positive fixed order costs implies that the mathematical problem is of impulse control type.  The models under consideration are restricted to those having {\em continuous inventory levels,}\/ such as the amount of water behind a dam, the amount of natural gas in a storage facility or the excess capacity of power generation plants.  More precisely, we analyze state-dependent stochastic inventory models.  Such models capture, for instance, the effect visible inventory has on sales and demand.  The review article by \cite{urba:05} provides a comprehensive survey of deterministic state-dependent inventory models up to 2004.  Urban's review also discusses extensions of a basic model and suggests a taxonomy of general models; see also the survey article by \cite{bakk:12}.  Among the many recent articles devoted to inventory theory and shelf-space management, we mention the papers by \cite{berm:06}, \cite{jain:08}, and \cite{baro:11}.  These articles and many purely mathematically oriented articles (see later citations) as well as specific shelf-space management problems (both deterministic and stochastic) have stimulated our research on state-dependent stochastic inventory models.  

We model the inventory processes (in the absence of orders) as solutions to a stochastic differential equation
\begin{equation} \label{dyn}
\d X_0(t) = \mu(X_0(t))\d t + \sigma(X_0(t))\d W(t)
\end{equation}
taking values in an interval ${\mathcal I} = (a,b)  $ with $-\infty \le a < b \le \infty$; negative values of $X_0(t)$ represent back-ordered inventory.  In contrast to the literature in general, we do not impose a sign condition on the drift coefficient $\mu$ but rather require $a$ and $b$ to be certain types of boundaries.  This paper takes advantage of the boundary classification theory for one-dimensional diffusions (see {\em e.g.} Chapter 15 of \cite{KarlinT81}).

An ordering policy $(\tau,Y)$ is a sequence of pairs $\{(\tau_k,Y_k): k \in \N\}$ in which $\tau_k$ denotes the (random) time at which the $k^{th}$ order is placed and $Y_k$ denotes its size.  Since order $k+1$ cannot be placed before order $k$, $\{\tau_k: \in \N\}$ is an increasing sequence of times.  The inventory level process $X$ resulting from an ordering policy $(\tau,Y)$ therefore satisfies the equation
\begin{equation} \label{controlled-dyn}
X(t) = X(0-) + \int_0^t \mu(X(s))\d s + \int_0^t \sigma(X(s))\d W(s) + \sum_{k=1}^\infty I_{\{\tau_k \leq t\}} Y_k.
\end{equation}
Note the initial inventory level $X(0-)$ may be such that an order is placed at time $0$ resulting in a new inventory level at time $0$; this possibility occurs when $\tau_1 = 0$.  Also observe that $X(\tau_k-)$ is the inventory level just prior to the $k^{th}$ order being placed while $X(\tau_k)$ is the level with the new inventory.  Thus, this model assumes that orders are filled instantaneously.

Let $(\tau,Y)$ be an ordering policy and $X$ the resulting inventory level process.  Let $c_0$ and $c_1$ denote the holding/back-order cost rate and ordering cost functions, respectively.  We assume there is some constant $k_1 > 0$ such that $c_1 \geq k_1$; this constant represents the fixed cost for placing each order.  The goal of the inventory management is to minimize the long-term average expected holding/back-order plus ordering costs
\begin{equation} \label{obj-fn}
J(\tau,Y):= \limsup_{t\rightarrow \infty} t^{-1} \E\left[\int_0^t c_0(X(s))\d s + \sum_{k=1}^\infty I_{\{\tau_k \leq t\}} c_1(X(\tau_k-),X(\tau_k))\right].
\end{equation}

Stochastic control motivated by inventory models has a long history and is a common theme in the literature.  Early work on discrete time models may be found in the Stanford Studies (see, e.g., \cite{arro:58}); for an excellent survey, see \cite{girl:01}.  Some early work on continuous time models can be found in \cite{bather-66}, \cite{HarrisonST-83}, \cite{Harrison-85}, \cite{sulem-86} and the references therein. More recent work include \cite{FleiK-03}, \cite{Bens:05}, \cite{PresmanS-06}, \cite{Bens:10}, \cite{cadenillas-10}, \cite{Bens-11}, \cite{Bensoussan-13}, \cite{DaiY-13-average}, and \cite{DaiY-13-discounted}, to name just a few.  In \cite{Bensoussan-13}, a discrete inventory system is studied in which the unmet demand is lost and the excess inventory is subject to shrinkage.  In \cite{sulem-86}, \cite{DaiY-13-average}, and \cite{DaiY-13-discounted}, the inventory processes are modeled by drifted Brownian motion,  \cite{cadenillas-10} considers a mean-reverting Brownian motion while in the models considered in  \cite{Bens:05}, \cite{PresmanS-06}, and \cite{Benk:09}, compound Poisson demands are included in addition to the fluctuations modeled by drifted Brownian motion.  In these articles, the adjustments to inventory levels are represented by controlled impulse jumps (upward or downward) and each such impulse includes a fixed cost and a proportional cost.  In addition, a state-dependent holding and/or back-order cost is also included.  Often, the focus of these papers is to establish the optimality of  an $(s,S)$-policy under the long-run average or discounted cost minimization criteria.  For more literature on the optimality of $(s,S)$ policies for general discounted problems as well as new results, see \cite{HeYZ-15} and \cite{HelmesSZ-14}.

This paper's contributions begin by considering a general one-dimensional diffusion inventory model on some interval with general boundary behavior consistent with a positive demand (see \cndref{diff-cnd}) using general cost structure (see \cndref{cost-cnds}).  Further, it shows that the nonlinear optimization approach used on discounted problems in \cite{HelmesSZ-14} directly applies for a long-term average cost criterion.  For the long-term average problem, this method identifies two simple conditions (one each in \propref{prop-G-qvi} and \thmref{thm-verification}) such that an $(s,S)$ policy is optimal within a large class of ordering strategies.  For specific models, this optimality can be extended to all admissible ordering policies (see e.g., \sectref{sect:examples}).    

This paper and \cite{HelmesSZ-14} both minimize a nonlinear function of two variables $(y,z) \in {\cal R}$, called $F$ in each paper, and use the optimal value $F_*$ and an optimizing pair $(y_*,z_*)$ to determine the optimal value for all initial values.  In this way, the papers solve the respective problems similarly.  However, the two papers approach the stochastic problems differently.  \cite{HelmesSZ-14} derives the function $F$ from a careful analysis of a large class of admissible policies and then comments that the function $F$ is the cost for a related $(s,S)$ policy.  Moreover, the paper utilizes a linear programming imbedding to define the class of admissible policies as well as to determine the function $F$.  In contrast, this paper only involves stochastic analysis; it does not employ a linear programming formulation.  It begins with the careful analysis of the costs associated with an $(s,S)$ ordering policy, thus determining the function $F$.  It then specifies a more general class of ordering policies within which the optimal $(s,S)$ policy is again optimal.  The function $F$ does not arise from analysis of a large class of policies; rather optimality is reverse engineered from knowledge of the costs corresponding to $(s,S)$ ordering policies.  

This paper is organized as follows.  The next section formulates the problem and, in \propref{cor-g0-psi-fns}, defines two key functions and establishes their relation to the expected costs and expected hitting times of the solution $X_0$ of \eqref{dyn}.  \sectref{sect:soln} begins by examining the long-term behavior of $X$ under an $(s,S)$ policy, identifying the stationary distribution and representing the long-term average cost as the function $F$ of \eqref{eq-F-fn}.  Under the model conditions of \sectref{sect:form}, we prove existence of an  optimizing pair.  Using this solution, the function $G$ of \eqref{eq-G-fn} is defined and is proven to be a solution of the quasi-variational inequality (QVI) associated with the impulse control problem provided  \cndref{eq-AG-c0} is satisfied (see \remref{rem-qvi}).  Finally, we show that the optimal $(s,S)$ policy remains optimal within a large class of policies. We then examine a drifted Brownian motion model having a piecewise linear holding/back-order cost rate function $c_0$ and fixed plus proportional ordering cost function $c_1$ in \sectref{sect-dbm} and using a comparison result of \cite{HeYZ-15} show that the optimal $(s,S)$ policy is optimal over all admissible ordering policies.  However, for even a slight modification of the model, optimality of an $(s,S)$ policy in the class of admissible policies may no longer hold.  For example, for the special case of a Brownian motion model with negative drift and reflection at $0$, we introduce a class of non-Markovian policies, called ``delayed $(s,S)$ policies with trigger $\ss$'' in \sectref{sect:refl-dBm}.  We identify a simple sufficient condition under which each such policy has a lower cost than the corresponding $(s,S)$ policy.  Briefly in \sectref{sect:just-in-time}, we interpret reflection at $\{0\}$ as a just-in-time ordering policy and give a necessary and sufficient condition for the just-in-time policy to be less costly than the optimal $(s,S)$ policy.  \sectref{sect:gBM} then considers a geometric Brownian motion model with nonlinear $c_0$ and $c_1$ containing fixed plus state-dependent costs.  Optimality of an $(s,S)$ policy in the general class of admissible ordering policies (cf. \defref{admissible-class-A-defn}) is established by the comparison in \propref{prop-comparison-gbm2}.  We note that the comparison result in \propref{prop-comparison-gbm2} holds pathwise whereas the similar result in \cite{HeYZ-15} only holds in expectation.  The technical proof of \propref{prop-comparison-gbm2} is relegated to the appendix.  Sections \ref{case-1}-\ref{case-3} examine three variations for the parameters in the cost functions and determines the optimality or not of an $(s,S)$ policy.

\section{Formulation} \label{sect:form}
Let ${\mathcal I} = (a,b) \subseteq \R$.  In the absence of ordering, the inventory process $X_0$ satisfies \eqref{dyn} and is assumed to be a regular diffusion.  Throughout the paper we assume that the functions $\mu$ and $\sigma$ are continuous on $\I$ and that \eqref{dyn} is nondegenerate. The initial position of $X_0$ is taken to be $x_0$ for some $x_0 \in {\cal I}$.  Let $\{{\mathcal F}_t\}$ denote the filtration generated by $X_0$, augmented so that it satisfies the usual conditions.  The generator of $X_0$ is
\begin{equation} \label{A-def}
Af(x) = \mbox{$\frac{\sigma^2(x)}{2}$} f''(x) + \mu(x) f'(x)
\end{equation}
which is defined for all $f \in C^2({\mathcal I})$.

\begin{cnd} \label{diff-cnd}
\begin{description}
\item[(a)] Both the speed measure $M$ and the scale function $S$ of the process $X_0$ are absolutely continuous with respect to Lebesgue measure; that is, the scale function
$$S(x):= \int^{x} s(v) \d v, \qquad x \in \I ,$$
in which  $s(x)= \exp(-\int^{x} [2 \mu(v)/ \sigma^{2}(v)] \d v)$, and the speed measure
$$M[y,z]:= \int_y^z m(v) \d v, \qquad [y,z] \subset {\mathcal I} ,$$
where $m(x)= 1/[\sigma^{2}(x) s(x)]$ is the speed density for $x \in {\cal I} $.
\item[(b)] The left boundary, $a$, of ${\cal I}$ is attracting and the right boundary $b$ is non-attracting; see Definition~6.1 of Chapter~15 in \cite{KarlinT81}.
\end{description}
\end{cnd}

Since $\mu$ and $\sigma$ are continuous and $\sigma$ is non-degenerate in ${\cal I}$, $S(x) < \infty$ for $x \in {\cal I}$ and $M[y,z] < \infty$ for all $a < y < z < b$.  Associated with the scale function $S$ of \cndref{diff-cnd}, one can define the scale measure on the Borel sets of ${\mathcal I}$ by $S[y,z] = S(z) - S(y)$ for $[y,z] \subset {\mathcal I}$. We refer the reader to Chapter 15 in \cite{KarlinT81} for details about $S$ and $M$.  The scale measure $S$ and the speed measure $M$ are used to define two functions $g_0$ and $\zeta$ in \eqref{eq-g0-fn} and \eqref{eq-psi-fn} that are critical to the solution of the inventory control problem.

From the modeling point of view, \cndref{diff-cnd}(b) is reasonable since it essentially says that, in the absence of ordering, demand tends to reduce the size of the inventory and there is a positive probability that $X_0(t)$ converges to $a$ as $t\rightarrow \infty$.  In addition, it imposes a reasonable restriction on ``returns'' since it implies that the inventory level will never reach level $b$ solely by diffusion in finite time ($a.s.$).  According to Table~6.2 of Chapter~15 (p.~234) in \cite{KarlinT81}, $a$ may be a regular, exit or natural boundary point with $a$ being attainable in the first two cases and unattainable in the third.  In the case that $a$ is a regular boundary, its boundary behavior must also be specified as being either reflective or sticky (see Chapter 15 of \cite{KarlinT81}).  Again by Table~6.2 of Chapter~15 in \cite{KarlinT81}, $b$ is either a natural or an entrance boundary point and is unattainable from the interior in both cases.  In order to unify our presentation for models having these various types of boundary points, we define the state space of possible inventory levels to be 
$${\cal E} = \begin{cases}
(a,b), & \text{ if } a \mbox{ and } b \mbox{ are natural boundaries,} \\
[a,b), & \mbox{ if $a$ is attainable and $b$ is natural}, \\
(a,b], & \mbox{ if $a$ is natural and $b$ is entrance}, \\
[a,b], & \mbox{ if $a$ is attainable and $b$ is entrance}.
\end{cases} $$
When $a = -\infty$ and/or $b = \infty$, we require these boundaries to be natural.  Since orders increase the inventory level, define ${\cal R} = \{(y,z) \in {\cal E}^2: y < z\}$ in which $y$ denotes the pre-order and $z$ the post-order inventory levels, respectively, and let $\overline{\cal R} = \{(y,z) \in {\cal E}^2: y \leq z\}$.  Notice, in particular, the requirement that $x_0 \in {\cal I}$ so the initial inventory level is not at either boundary when these are in the state space ${\cal E}$.

We note that the scale function $S$ is a fundamental solution to the homogeneous equation $Af = 0$; the other fundamental solution is the constant function.

We now specify the allowable types of orders and their costs.  

\begin{defn} \label{admissible-class-A-defn}
An ordering policy $(\tau,Y):=\{(\tau_k,Y_k): k\in \N\}$ is said to be {\em admissible} if
 \begin{itemize}
  \item[(i)] $\{\tau_k:k\in \N\}$ is an increasing sequence of $\{{\cal F}_t\}$-stopping times, and
  \item[(ii)] for each $k \in \N$, $Y_k$ is nonnegative, ${\cal F}_{\tau_k}$-measurable and satisfies $X(\tau_{k}) \in {\cal E}$.
\end{itemize}
The class of admissible policies is denoted by ${\cal A}$.
\end{defn}

Turning to the cost functions, we impose the following standing assumptions throughout the paper.
\begin{cnd} \label{cost-cnds}
\begin{itemize}
\item[(a)] The holding/back-order cost function $c_0: {\mathcal  I} \rightarrow \R ^+$ is  continuous
and  inf-compact (for each $L > 0$, $\{x\in {\mathcal  I}: c_0(x) \leq L\}$ is compact).  Moreover,
\begin{equation} \label{c0-lim-at-a}
\text{ if } a \in {\cal E},\text{ then  }\lim_{x\rightarrow a} c_0(x) = \infty, \text{ and if } b \in {\cal E}, \text{ then } \lim_{x\rightarrow b} c_0(x) = \infty.
\end{equation}
In addition, for each $y \in {\cal I}$,
\begin{equation} \label{c0-M-integrable}
\int_y^b c_0(v)\, \d M(v)  < \infty
\end{equation}
and
\begin{equation} \label{infinite-dbl-intgrl-at-b}
\int_y^b \int_u^b c_0(v)\, \d M(v)\, \d S(u) = \infty.
\end{equation}
\item[(b)] The function $c_1:\overline{\cal R} \rightarrow \R^+$ is continuous with $c_1 \geq k_1 > 0$ for some constant $k_1$.  Furthermore, for each $y,z \in {\cal I}$ with $y \leq z$,
\begin{equation} \label{decr-cost}
c_1(y,z) \geq c_1(z,z)
\end{equation}
and for each $w,x,y,z \in {\mathcal  I}$ with $w \leq x \leq y \leq z$,
\begin{equation} \label{eq-c1-equal-displacement}
  c_{1}(w,z) + c_{1}(x,y) = c_{1}(w,y) + c_{1}(x,z).
\end{equation}
\end{itemize}
\end{cnd}
The model and cost assumptions place very mild conditions at the boundary $a$.  Requirement \eqref{c0-lim-at-a} follows immediately from $c_0$ being inf-compact when $a$ and $b$ are natural boundaries so this condition only affects models in which $X_0$ can reach $a$ in finite time with positive probability or in which the process starts at $b$ and immediately enters ${\cal I} $ without diffusing back to $b$.  
Notice also that \eqref{c0-lim-at-a} implies that the holding/back-order costs are very expensive when the inventory level is near either boundary.  In particular, when $a$ is either an exit or a sticky boundary point, any admissible ordering policy $(\tau,Y)$ which allows the inventory level to spend a positive amount of time at $a$ will incur an infinite cost.  

The integrability condition of $c_0$ with respect to $M$ in \eqref{c0-M-integrable} is imposed on the model since the model does not allow the decision maker the option of reducing the inventory and therefore he or she cannot control how large the inventory level becomes.  Should this condition fail, an infinite long-term average cost would accrue simply due to the process diffusing toward the boundary $b$.  As will be seen, the non-integrability condition \eqref{infinite-dbl-intgrl-at-b} and assumption \eqref{c0-lim-at-a} on $c_0$ are used in \propref{F-optimizers} to establish the existence of an optimizing pair $(y_*,z_*) \in {\cal R}$ for a nonlinear function $F$.  This pair is then shown to provide the levels for an optimal $(s,S)$-ordering policy in Theorem  \ref{thm-verification}.

With regard to the ordering cost function $c_1$, \eqref{decr-cost} requires that it be no less expensive to order a positive quantity to achieve an inventory level $z$ than it is to order nothing; note that actively ordering nothing incurs at least the fixed cost $k_1$ so differs from not ordering.  Equation \eqref{eq-c1-equal-displacement} of \cndref{cost-cnds}(b) is a technical condition that is used in the proof of \propref{prop-G-qvi}.  This condition is not very restrictive.  For example, let $H \in C({\cal E})$ be monotone increasing.  Then defining $c_1$ on $\overline{\cal R}$ by $c_1(y,z) = k_1+H(z)-H(y)$ will satisfy \cndref{cost-cnds}(b).  Notice that this expression for $c_1$ allows great flexibility in the ordering cost structure.  In particular, the ``per unit'' cost may vary depending on the choices of $y$ and $z$ and, in the case that the function $H$ is also concave, the model allows for economies of scale in the ordering costs.  The geometric Brownian motion inventory model of \sectref{sect:gBM} adopts this structure for $c_1$. 

By selecting $y=x$ in \eqref{eq-c1-equal-displacement} and noting $c_{1}(x,x) \ge k_{1}> 0$, we obtain
\begin{equation}
\label{eq-c1-sub-add}
  c_1(w,z) < c_1(w,x) + c_1(x,z), \quad \forall w < x < z.\end{equation}
The triangle inequality \eqref{eq-c1-sub-add} for $c_1$ requires that the ordering costs be such that it is more expensive to place two orders at the same time than to combine them into a single order.
Further by setting $y=x$ in  \eqref{eq-c1-equal-displacement}, it follows from \eqref{decr-cost} that for any $w \leq x \leq z$, the ordering cost function $c_1$ is monotone decreasing in the first argument:
\begin{equation}
\label{eq-c1-monotone}
c_{1}(w,z) \ge c_{1}(x,z).
\end{equation}

When $b$ is an entrance boundary, $M[y,b) < \infty$ for each $y \in {\cal I}$ (see Table 7.1 of Chapter 15 in \cite{KarlinT81}); when $b$ is a natural boundary $M[y,b)$ may be either finite or infinite.  However, the assumptions on the holding/back-order cost function implies this mass is finite, which in turn implies the existence of stationary measures for $(s,S)$ policies (see \propref{sS-stationary}).

\begin{lem} \label{M-finite}
Assume $b$ is a natural boundary point and \cndref{cost-cnds} holds.  Then $M[y,b) < \infty$ for each $y \in {\cal I}$.
\end{lem}

\begin{proof}
By the inf-compactness of $c_0$ in \cndref{cost-cnds} (a), the set  $\set{x \in \I: c_{0}(x) \le 1}$ is compact in $\I$. Observe that the speed measure $M$ is finite on any compact subset of ${\cal I}$ since $m$ is continuous and hence bounded on the compact set.  Therefore by \eqref{c0-M-integrable}, we have
\begin{equation} \label{eq-M-c0-integral} \begin{aligned}M[y,b) & = \int_{y}^{b} \left[ I_{\{v\in \I: \ c_{0}(v) \le 1\}} +  I_{\{v\in \I:\  c_{0}(v) > 1\}}\right] \d M(v) \\ & \le M (\set{v\in \I: c_{0}(v) \le 1} \cap [y, b))  + \int_{y}^{b} c_{0}(v) \d M(v) < \infty. \end{aligned}
\end{equation}
\end{proof}

We now seek to derive the general representation result \eqref{h-integral-repr} below.  We denote by $X_0$ the diffusion process specified by \eqref{dyn} with initial condition $X_0(0)=x_0$ and, for $a < l  \le  x_0 \le r  < b  $,  define
$$\tau_{l}:=  \inf\set{t \ge 0: X_0(t) = l} ,  \quad \tau_{r}:=  \inf\set{t \ge 0: X_0(t) = r}, \quad \mbox{and} \quad \tau_{l,r}= \tau_{l} \wedge \tau_{r}.$$
Let $h: \I\rightarrow \R^+$ be continuous. Then by the representation in Chapter 15 (see p.\ 197) of \cite{KarlinT81} together with the monotone convergence theorem, we have
\begin{equation} \label{eq-h-KT}
\begin{array}{rcl} \displaystyle
\E_{x_0}\left[\int_{0}^{\tau_{l,r}} h(X_{0}(s))\d s \right] &=& \displaystyle 2 u(x_0) \! \int_{x_0}^{r} [S(r)- S(v)] h(v) \d M(v) \\
& & \displaystyle+\; 2(1-u(x_0)) \!\int_{l}^{x_0} [S(v)- S(l)] h(v) \d M(v),
\end{array}
\end{equation}
where $u(x_0) = \P_{x_0}(\tau_{r} < \tau_{l})= \frac{S(x_0) - S(l)}{S(r) -S(l)}$.  (Note this paper uses $a$ and $b$ as the boundary points with $l$ and $r$ being interior points whereas \cite{KarlinT81} denotes the boundary points as $l$ and $r$ with $a$ and $b$ being interior points.  The formulas from \cite{KarlinT81} have been written using our notation.)

\begin{prop}\label{prop-2.4-h-integral}
Assume Condition \ref{diff-cnd} and let $h:\I \rightarrow \R^+$ be continuous and satisfy
\begin{equation}
\label{eq-h-integrable}
  \int_{y}^{b} h(v) \d M(v)  < \infty \text{ for every } y \in \I.
\end{equation}
Then for any $x_{0}\in \I$,  
\begin{equation} \label{h-integral-repr}
\begin{array}{rcl} \displaystyle 
\E_{x_0}\left[\int_{0}^{\tau_{l}} h(X_0(s))\d s \right]  &=& \displaystyle  2 \int_{l}^{x_0} [S(v) - S(l)] h(v) \d M(v) \\
& & \displaystyle \qquad +\; 2 [S(x_0) - S(l) ] \int_{x_0}^{b} h(v) \d M(v).
\end{array}
\end{equation}
\end{prop}

\begin{proof}
Note  that $b$ being non-attracting means that it is unattainable from the interior and hence $\tau_{r} \to \infty$ $a.s.$\ as $r \nearrow b$. Then it follows from the monotone convergence theorem that the left-hand side of \eqref{eq-h-KT} converges to $\E_{x_0}[\int_{0}^{\tau_{l}} c_{0}(X(s))\d s ] $ as  $r \nearrow b$. On the other hand, we rewrite the right-hand side of \eqref{eq-h-KT} as
\begin{equation}
\label{eq-rhd}
 \begin{aligned}
 2  &  \int_{l}^{x_0} [S(v)- S(l)] h(v) \d M(v)  \\ & + 2 [S(x_0) - S(l)] \int_{x_f0}^{b} \frac{ S(r) -S(v)}{S(r)- S(l)} I_{\{ v \le r\}} h(v) \d M(v)  \\   & -2 [S(x_0) -S(l)] \int_{l}^{x_0} \frac{S(v) - S(l)}{S(r)  - S(l)} h(v) \d M(v).
\end{aligned}
  \end{equation}
For the second term of \eqref{eq-rhd}, the integrand $\frac{ S(r) -S(\cdot)}{S(r)- S(l)} I_{\{\cdot \le r\}} h(\cdot) $ is nonnegative, bounded above by and converges pointwise to the integrable function $h(\cdot)$ as $r \nearrow b$. Thus by \eqref{eq-h-integrable} and the dominated convergence theorem, the second term converges to
$$2 [S(x_0) - S(l)] \int_{x_0}^{b} h(v) \d M(v)$$
as $r \nearrow b$.
Similarly, the third term of \eqref{eq-rhd} converges to $0$ as $r \nearrow b$. Therefore as $r \nearrow b$, the right hand side of  \eqref{eq-h-KT}  converges to
$$2  \int_{l}^{x_0} [S(v)- S(l)] h(v) \d M(v)   + 2  [S(x_0) - S(l)] \int_{x_0}^{b} h(v) \d M(v),$$
establishing \eqref{h-integral-repr}.
\end{proof}

\begin{prop}\label{cor-g0-psi-fns}
Assume Conditions \ref{diff-cnd} and \ref{cost-cnds} hold.  Fix $C \in {\cal I}$ arbitrarily.  Define the functions $g_0$ and $\zeta$ on ${\cal E}$ by
\setlength{\arraycolsep}{0.5mm}
\begin{eqnarray}
\label{eq-g0-fn}
 g_{0}(x) = g_{0}(x; C)&:=& \int_{C}^{x} \int_{u}^{b} 2 c_{0}(v) \d M(v) \d S(u), \qquad \\
\label{eq-psi-fn}
\zeta(x) = \zeta(x;C)&:=& 2 \int_{C}^{x} M[u, b) \d S(u).
\end{eqnarray}
Then for any $(y,z) \in {\cal R}$ with $a < y < z < b$,
\begin{equation}
\label{eq-path-integral-c0-yz}
\E_{z}\left[ \int_{0}^{\tau_{y}}  c_{0}(X_0(s)) \d s\right]  = g_0(z) - g_0 (y),  \quad \mbox{ and } \quad \E_{z}[\tau_{y}] = \zeta(z) - \zeta(y).
\end{equation}
\end{prop}

\begin{proof}
\lemref{M-finite} and \cndref{cost-cnds}(a) implies that \eqref{eq-h-integrable} is satisfied with $h\equiv 1$ and with $h = c_0$, respectively.  From \eqref{h-integral-repr}, we  have  for any $(y,z) \in {\cal R}$ with $a < y < z <b$,
  \begin{align*}\E_{z}& \left[ \int_{0}^{\tau_{y}}  c_{0}(X_0(s)) \d s\right] \\ & = 2  \int_{y}^{z} [S(v)- S(y)] c_{0}(v) \d M(v)   + 2  [S(z) - S(y)] \int_{z}^{b} c_{0} (v) \d M(v) \\
  & =  2 \int_{y}^{z}\left[ \int_{y}^{v} \d S(u)\right]c_{0}(v) \d M(v) + 2  \int_{y}^{z} \left[\int_{z}^b c_{0} (v) \d M(v) \right] \d S(u) \\
  & =  2 \int_{y}^{z}\left[  \int_{u}^{z} c_{0}(v) \d M(v)\right]  \d S(u)  +  2  \int_{y}^{z} \left[\int_{z}^b c_{0} (v) \d M(v) \right] \d S(u)  \\
  & =   \int_{y}^{z}\left[  \int_{u}^{b} 2 c_{0}(v) \d M(v)\right]  \d S(u) \\
  & = g_{0}(z) - g_{0}(y),
\end{align*}  establishing the first identity in \eqref{eq-path-integral-c0-yz}.  The second identity of \eqref{eq-path-integral-c0-yz} follows from a similar argument.
\end{proof} 

\begin{corollary}\label{cor-b-entrance} 
For any $\ell \in \mathcal I$, we have 
\begin{equation}\label{eq1-cor-b-entrance}
\lim_{x\nearrow b}\E_{x} \biggl[ \int_{0}^{\tau_{\ell}} c_{0}(X_{0}(s)) \d s\biggr] = \infty,
\end{equation} where $\tau_{\ell}: = \inf\{t \ge 0: X_0(t) = \ell \}.$
In particular, if $b$ is an entrance boundary point, then  
\begin{equation}\label{eq2-cor-b-entrance}
\E_{b} \biggl[ \int_{0}^{\tau_{\ell}} c_{0}(X_{0}(s)) \d s\biggr] = \infty,
\end{equation} 
\end{corollary}
\begin{proof}
For $\ell < x < b$, let $X_{0}(t; x)$ denote the solution of \eqref{dyn} with initial condition $X_0(0;x) = x$ and define $\tau^{x}_{\ell} : = \inf\{t\ge 0: X_{0}(t; x) \le \ell \}.$  By Proposition \ref{cor-g0-psi-fns}, we have 
\begin{equation}\label{eq1-b-entrance}
 \E\biggl[ \int_{0}^{\tau^{x}_{\ell}} c_{0}(X_{0}(s;x)) \d s\biggr]=   \E_{x}  \biggl[ \int_{0}^{\tau_{\ell}} c_{0}(X_{0}(s)) \d s\biggr] = g_{0}(x) - g_{0}(\ell). 
\end{equation}  
As $x\nearrow b$, \eqref{infinite-dbl-intgrl-at-b} implies that the rightmost expression   of \eqref{eq1-b-entrance} converges to infinity; from which \eqref{eq1-cor-b-entrance} follows. 

Now we analyze the limiting behavior of the leftmost expression of   \eqref{eq1-b-entrance} when $b$ is an entrance boundary.  Notice that $\tau^{x}_{\ell} \nearrow \tau^{b}_{\ell}$ as $x \nearrow b$ so it follows from the monotone convergence theorem that 
\begin{displaymath}
 \lim_{x \nearrow b}\E\biggl[ \int_{0}^{\tau^{x}_{\ell}} c_{0}(X_{0}(s;x)) \d s\biggr]  =  \E\biggl[ \int_{0}^{\tau^{b}_{\ell}} c_{0}(X_{0}(s;b)) \d s\biggr] = \E_{b} \biggl[ \int_{0}^{\tau_{\ell}} c_{0}(X_{0}(s)) \d s\biggr]. 
\end{displaymath} This implies \eqref{eq2-cor-b-entrance} and completes the proof. 
\end{proof}
\begin{rem}
Recall from equation (3.9) on p. 214 of \cite{KarlinT81} that the generator $A$ defined in \eqref{A-def} can be rewritten as
\begin{equation} \label{A-equiv-defn}
Af(x) = \frac{1}{2} \frac{\d}{ \d M}\left[ \frac{\d f(x)}{\d S}\right], \quad \forall f \in C^2 (\I).
\end{equation}
Since $c_{0}$ and the constant function $1$ are  continuous, it follows immediately that the functions  $g_0$ and $\zeta$ defined in \eqref{eq-g0-fn} and \eqref{eq-psi-fn} are twice continuously differentiable and respectively satisfy
\begin{equation} \label{A-g0=c0}
 A g_0(x) = - c_0(x), \quad x \in \I ,
\end{equation}
and
\begin{equation} \label{A-psi=1}
A\zeta (x) = -1, \quad x \in \I .
\end{equation}
Also both $g_{0}$ and $\zeta$ are increasing functions.
\end{rem}

\section{Solution of the Inventory Control Problem} \label{sect:soln}
We break the analysis of the solution to the inventory problem into two parts.  The first part examines $(s,S)$ ordering policies while the second extends the optimality of an $(s,S)$ policy to a large class of admissible ordering policies.

\subsection{Examination of $(s,S)$ ordering policies} \label{subsect:sS}
Our model only allows orders which increase the inventory level.  As a result, the manager is able to control how low the inventory level becomes so it is not necessary to require finiteness of the speed measure near $a$.  The opportunity to order may result in processes having stationary distributions, as the next proposition demonstrates.

\begin{prop} \label{sS-stationary}
Assume Conditions \ref{diff-cnd} and \ref{cost-cnds}.  Let $x_0 \in {\mathcal  I}$ and $(y,z) \in {\mathcal  R}$.   Set $\tau_0=0$ and define the ordering policy $(\tau,Y)$ by
\begin{equation} \label{sS-tau-def}
\left\{\begin{array}{l} 
\tau_k = \inf\{t > \tau_{k-1}: X(t-) \leq y\}, \\
Y_k = z-X(\tau_k-),
\end{array} \right.
\quad  k \geq 1,
\end{equation}
in which $X$ is the inventory level process satisfying \eqref{controlled-dyn} with this ordering policy and $X(0-) = x_0$.  Then $X$ has a stationary distribution with density $\pi$ on ${\mathcal  I}$ given by
\begin{equation} \label{eq-pi-defn}
\pi(x) = \left\{\begin{array}{cl}
0, & \quad a < x \leq y, \rule[-8pt]{0pt}{8pt} \\
\displaystyle 2\kappa\, m(x)\, S[y,x], &  \quad  y \leq x \leq z, \rule[-8pt]{0pt}{8pt} \\
\displaystyle 2\kappa\, m(x)\, S[y,z], &  \quad  z \leq x < b,
\end{array}\right.
\end{equation}
in which
\begin{equation} \label{norm-const}
\kappa =  \( \int_{y}^{z} 2S[y,x] \d M(x) + 2S[y,z] M[z,b)\)^{-1} = \frac{1}{\zeta(z)-\zeta(y)}.
\end{equation}
is the normalizing constant.  Moreover, the constant $\kappa$ gives the expected frequency of orders. 
\end{prop}

\begin{rem}
The ordering policy \eqref{sS-tau-def} uses the levels $s=y$ and $S=z$ for an $(s,S)$-ordering policy.  Proposition \ref{sS-stationary} therefore establishes that every $(s,S)$ policy has a stationary distribution.  Note, in particular, that when $a\in {\cal E}$ and $y=a$, the policy \eqref{sS-tau-def} orders at the first hitting time of the process at $a$ so the boundary behavior of $X$ does not occur.  Also notice that when $y=a$, the definition of the density $\pi$ on $a < x \leq y$ is on the empty set and similarly when $b \in {\cal E}$ and $z=b$, the definition on $z \leq x < b$ is on the empty set.

Observe also that $\kappa^{-1} = \zeta(z) - \zeta(y)$ gives the expected cycle length, in agreement with \eqref{eq-path-integral-c0-yz}.
\end{rem}

\begin{proof}
We derive the density $\pi$ in the most complicated case of $a < y < z < b$.  

Under the policy $(\tau,Y)$, the process $X$ jumps to $z$ whenever it hits $y$ (or at the initial time if $x_0 < y$).  Between jumps, the process evolves according to the generator $A$ of $X$ in \eqref{A-def} but at the times $\{\tau_k: k \geq 1\}$, the jump operator $Bf(y,z) = f(z) - f(y)$ governs the process.  The stationary density $\pi$ and jump frequency $\kappa$, if they exist, satisfy the stationarity condition
\begin{equation} \label{stat-cnd}
\int_{\mathcal  I} Af(x)\, \pi(x)\d x + Bf(y,z) \kappa = 0, \qquad \forall f \in C^2_b({\mathcal  I}),
\end{equation}
(see \cite{kurt:01} for existence of a stationary process corresponding to the identity \eqref{stat-cnd} in a more general setting).
We solve this identity.

First observe that the ordering policy implies that $X$ is never in the range $(a,y)$ so $\pi = 0$ on this interval and we seek the density on the interval $[y,b)$; the evaluations of the functions below at $b$ are to be interpreted as being the limit as $r \nearrow b$ of the functions evaluated at $r$.  Next, due to the jumps in the process, we define $\pi$ in a piecewise manner on the sets $[y,z]$ and $(z,b)$.  Let $\pi_1$ and $\pi_2$ denote these respective parts of the density.  The condition \eqref{stat-cnd} then takes the form
$$\int_y^z Af(x) \pi_1(x)\d x + \int_z^b Af(x)\, \pi_2(x)\d x + Bf(y,z) \kappa = 0, \qquad \forall f \in C^2_b({\mathcal  I}).$$
Let $f \in C^2_b({\mathcal  I})$.  Using \eqref{A-def}, a formal application of integration by parts yields
\begin{eqnarray} \label{by-parts-id} \nonumber
0 &=& \int_y^z f(x)\left[\mbox{$\frac{1}{2}$} (\sigma^2 \pi_1)''(x) - (\mu \pi_1)'(x)\right]\!\d x  +\; \int_z^b f(x)\left[\mbox{$\frac{1}{2}$} (\sigma^2 \pi_2)''(x) - (\mu \pi_2)'(x)\right]\!\d x \\ \nonumber
& &+\; \mbox{$\frac{1}{2}$} \sigma^2(z)\pi_1(z) f'(z) - \mbox{$\frac{1}{2}$} \sigma^2(y)\pi_1(y) f'(y) + \mbox{$\frac{1}{2}$} \sigma^2(b)\pi_2(b) f'(b) - \mbox{$\frac{1}{2}$} \sigma^2(z)\pi_2(z) f'(z) \\ \nonumber
& & +\; [\mu(z)\pi_1(z) - \mbox{$\frac{1}{2}$} (\sigma^2\pi_1)'(z)]f(z) - [\mu(y)\pi_1(y) - \mbox{$\frac{1}{2}$} (\sigma^2\pi_1)'(y)] f(y) \\
& & +\; [\mu(b)\pi_2(b) - \mbox{$\frac{1}{2}$} (\sigma^2\pi_2)'(b)]f(b) - [\mu(z)\pi_2(z) - \mbox{$\frac{1}{2}$} (\sigma^2\pi_2)'(z)] f(z) \\ \nonumber
& & +\; f(z) \kappa - f(y) \kappa.
\end{eqnarray}
We therefore seek functions $\pi_1$ on $[y,z]$ and $\pi_2$ on $(z,b)$ which are solutions to the differential equation
\begin{equation} \label{stat-de}
\mbox{$\frac{1}{2}$} (\sigma^2 \pi)'' - (\mu \pi)' = 0
\end{equation}
such that \eqref{by-parts-id} holds for each $f$.  These conditions require
\begin{equation}
\label{eq-pi-bd-conds}
\begin{array}{lcl}
0 &=& \lim_{r\rightarrow b} \sigma^2(r) \pi_2(r), \\
0 &=& \lim_{r\rightarrow b} [\mu(r) \pi_2(r) - \mbox{$\frac{1}{2}$} (\sigma^2 \pi_2)'(r)], \\
0 &=& \pi_1(z) - \pi_2(z), \\
0 &=& \mu(z) \pi_1(z) - \mbox{$\frac{1}{2}$} (\sigma^2\pi_1)'(z) - \mu(z) \pi_2(z) + \mbox{$\frac{1}{2}$} (\sigma^2\pi_2)'(z) + \kappa, \\
0 &=& \pi_1(y), \\
0 &=& \mu(y) \pi_1(y) - \mbox{$\frac{1}{2}$} (\sigma^2 \pi_1)'(y) + \kappa.
\end{array}
\end{equation}
By Equation (5.34) of \cite{KarlinT81} (see p.\ 221), the general solutions to \eqref{stat-de} on the respective intervals are
\begin{equation} \label{pi-def}
\pi_1(x) = K_1 m(x) S(x) + K_2 m(x), \quad \mbox{ and } \quad \pi_2(x) = K_3 m(x) S(x) + K_4 m(x),
\end{equation}
in which $K_1, K_2, K_3, K_4$ are arbitrary constants.  Observe that the function $\pi_2$ defined in \eqref{pi-def}  satisfies $\frac{1}{2}(\sigma^2\pi_2)'(x) - \mu(x)\pi_2(x) = \frac{K_3}{2}$.  Thus the second condition of \eqref{eq-pi-bd-conds} implies that $K_3 = 0$ for the density $\pi_2$ on $(z,b)$ and we have
\begin{equation} \label{pi2-form}
\pi_2(x) =  K_4 m(x), \qquad z < x < b
\end{equation}
for some value of $K_4$.  

Consider the first boundary condition of \eqref{eq-pi-bd-conds}.  Observe that $\sigma^2(r) \pi_2(r) = \frac{K_4}{s(r)}$.  When $b< \infty$, $b$ being non-attracting in \cndref{diff-cnd} implies that $S[y,b) = \infty$ (see Table 6.2 in Chapter 15 of \cite{KarlinT81}).  Thus $s(r) \rightarrow \infty$ as $r \rightarrow b$ and the first condition of \eqref{eq-pi-bd-conds} holds.  When $b = \infty$, it is a natural boundary and this condition follows from \lemref{M-finite}.

Next, the fifth condition of \eqref{eq-pi-bd-conds}  relates the constants $K_1$ and $K_2$ by requiring
$$K_1 m(y) S(y) + K_{2} m(y) =0.$$
So  $K_{2} = -K_{1} S(y)$.  Using this expression for $K_2$ in $\pi_1$ results in
\begin{equation}
\label{pi1-form}
\pi_1(x) = K_1   m(x) [S(x) - S(y)], \quad y \leq x  \leq z.
\end{equation}

The continuity at $z$ of $\pi_1$ and $\pi_2$ in the third condition of \eqref{eq-pi-bd-conds} gives
$K_1 m(z) [S(z) - S(y)] = K_4 m(z)$ or $K_4 = K_1[S(z) - S(y)]$.  Hence 
$$\pi_2(x) = K_1 m(x) [S(z)-S(y)].$$

At this point, we conclude that the density $\pi$ is defined piecewisely  by \eqref{eq-pi-defn}.  To determine the constant $K_1$ of \eqref{eq-pi-defn}, we use the fact that $\pi$ is a probability measure so $K_1$ becomes the normalizing constant.  Note that
$$\begin{aligned}
\int_{y}^{z} S[y,x] \d M(x) +  S[y,z] M[z,b) & =  \int_{y}^{z} M[u, z] \d S(u) + \int_{y}^{z} M[z,b) \d S(u) \\ &= \int_{y}^{z}  M[u,b) \d S(u) \\ & = \mbox{$\frac{1}{2}$}(\zeta(z) -\zeta(y))
\end{aligned}$$
so $K_1 = \frac{2}{\zeta(z) - \zeta(y)}.$

Turning to the determination of $\kappa$, observe $\frac{1}{2}(\sigma^2\pi_2)'(x) - \mu(x)\pi_2(x) = 0$ and $\frac{1}{2}(\sigma^2\pi_1)'(x) - \mu(x)\pi_1(x) = \frac{K_1}{2}$, where $\pi_1$ is defined by \eqref{pi1-form}.  Thus both the fourth and sixth conditions of \eqref{eq-pi-bd-conds} imply that $\kappa = \frac{K_1}{2} = (\zeta (z) - \zeta (y))^{-1}$.

Finally we note that the functions $\pi_1$ and $\pi_2$ are sufficiently smooth for the integration by parts argument to be valid which establishes the result.  The density $\pi$ is therefore  given by \eqref{eq-pi-defn}.
\end{proof}

\begin{rem}
Suppose $a$ is a reflecting boundary point; that is, $M(a,x] < \infty$ for each $x \in {\cal I}$ and $M(\{a\}) = 0$. By \lemref{M-finite}, $M(a,b) < \infty$ and the uncontrolled process $X_0$ of \eqref{dyn} has a stationary distribution.  By equation (5.34) on p.\ 221 of \cite{KarlinT81}, the stationary density has the form
$$\pi(x) = K_1 m(x) S(x) + K_2 m(x)$$
and the constants $K_1$ and $K_2$ are such that $\pi \geq 0$ and $\int_a^b \pi(x)\, \d x = 1$.  The stationarity condition for the reflected process $X_0$ (compare with \eqref{stat-cnd}) is
\begin{equation} \label{stat-cnd-refl}
\int_{\mathcal  I} Af(x)\, \pi(x)\d x + \kappa f'(a) = 0, \qquad \forall f \in C^2_b({\mathcal  I})
\end{equation}
and the same type of analysis as in the proof of \propref{sS-stationary} determines that $K_1 = 0$ so the stationary density is the normalized density $m$ of the speed measure.  Moreover, $\kappa = \lim_{t\rightarrow \infty} t^{-1} \E_{x_0}[L_a(t)]$ in which $L_a$ denotes the local time process of $X_0$ at $\{a\}$.
\end{rem}

Next, we derive a representation of the long-term average cost of an arbitrary $(s,S)$-ordering policy; let $s=y$ and $S=z$ and define $(\tau,Y)$ by \eqref{sS-tau-def}.  Using the strong Markov property observe that the inventory process $X$ associated with such a policy is a (possibly delayed) renewal process.  Except for the first cycle of the process when $x_0 > y$, each cycle starts at $z$ and lasts until the next hitting time of $X$ at $y$ at which time $X$ is reset to $z$ and the evolution starts afresh.  Thus $X$ consists of independent copies of the uncontrolled process $X_0$ (satisfying \eqref{dyn}) between renewal times.  

Observe that when $b$ is an entrance boundary, \corref{cor-b-entrance} implies that each $(y,b)$ ordering policy incurs an infinite cost.  Since we are interested in minimizing the cost, we may drop such ordering policies from further consideration.  

In light of Conditions~\ref{diff-cnd} and \ref{cost-cnds} and the renewal structure for the process $X$ associated with a $(y,z)$ ordering policy, the assumptions of the renewal-reward theorem are satisfied.  As a result, the representation of the long-term average cost can now be given.

\begin{prop}\label{prop-sS-cost}
Assume Conditions \ref{diff-cnd}  and \ref{cost-cnds} hold.  Fix $x_0 \in {\cal I}$, let $(y,z) \in \mathcal R$ with $z < b$ and let $X$ denote the inventory process satisfying \eqref{controlled-dyn} with $X(0-)=x_0$.  Then for the $(\tau, Y)$ policy defined in \eqref{sS-tau-def}, we have
\begin{equation}\label{eq-c0-integral-slln}
 \lim_{t\rightarrow \infty} \frac{1}{t} \int_{0}^{t} c_{0}(X(s))\d s = \frac{Bg_0(y,z)}{B\zeta(y,z)} \;\;(a.s.)
\end{equation} 
and, more importantly,
\begin{equation}
\label{eq-sS-cost}
J(\tau, Y) = \frac{c_{1}(y,z) + Bg_{0}(y,z)}{B\zeta(y,z)}.
\end{equation}
\end{prop}
\begin{proof}
By Proposition \ref{sS-stationary}, the controlled process $X$ is ergodic with stationary density $\pi$
given in \eqref{eq-pi-defn}. Moreover, due to the definitions of $g_0$ and $\kappa$ in \eqref{eq-g0-fn} and \eqref{norm-const}, respectively,   we compute
\begin{align*}
 \int_{\I} c_{0} (v) \pi(v) \d v & =  2\kappa \int_{y}^{z} c_{0}(v) m(v) S[y,v]\d v + 2\kappa \int_{z}^{b} c_{0}(v) m(v) S[y,z]\d v \\
   & = 2\kappa \int_y^z c_0(v) \int_y^v \d S(u) \d M(v) +  2\kappa \int_{z}^{b} c_{0}(v)  \int_y^z \d S(u) \d M(v) \\
   &  = 2\kappa \int_y^z \int_u^z c_0(v) \d M(v) \d S(u) + 2\kappa \int_y^z \int_z^b c_0(v) \d M(v) \d S(u) \\
    &  = \kappa \int_y^z \int_u^b 2 c_0(v) \d M(v) \d S(u) \\
    & = \frac{Bg_0(y,z)}{B\zeta(y,z)} < \infty.
\end{align*}

Observe that when $x_{0} > y$, the first order time $\tau_{1 } > 0$ and  the first cycle cost $\int_{0}^{\tau_{1}} c_{0}(X(s))\, \d s$ is independent of the later cycle costs; furthermore, by \propref{cor-g0-psi-fns}, the first cycle cost has finite expectation.  However, when $x_0 \neq z$, then $X(0) \neq X(\tau_k)$ for $k=1,2,3,\ldots$ so the process $X$ on the interval $[0,\tau_1)$ has a different distribution than $X$ on each later cycle $[\tau_k,\tau_{k+1})$.  When $x_{0}\le y$, the first order time $\tau_{1} =0$ so $X(0) = z = X(\tau_k)$ for each $k$.  By a similar argument as for Theorem 9 of Section~IV.5 in \cite{mand:68},  the random variables $\left\{\int_{\tau_k}^{\tau_{k+1}} c_0(X(s)) \d s: k=1,2,3, \ldots\right\}$, are mutually independent and identically distributed with finite expectation and therefore is a renewal-reward process.  Notice that the sequence $\left\{\int_{\tau_k}^{\tau_{k+1}} c_0(X(s)) \d s: k=0,1,2,3, \ldots\right\}$ forms a delayed renewal-reward process regardless of the relation of $x_0$ to $y$, and hence \eqref{eq-c0-integral-slln} follows from the elementary renewal-reward theorem. In addition, the delayed renewal-reward theorem implies that 
\begin{equation} \label{lta-running-cost}
\lim_{t\rightarrow \infty} \frac{1}{t}\E\left[ \int_{0}^{t} c_{0}(X(s))\d s \right]= \frac{Bg_0(y,z)}{B\zeta(y,z)}.
\end{equation}

Turning to the ordering costs, recall $\kappa = 1/ B\zeta(y,z)$ gives the long term expected frequency of interventions for the $(\tau, Y)$-policy in \eqref{sS-tau-def}.  As a result, we have from  \eqref{eq-path-integral-c0-yz} that  
$$\begin{aligned} 
\lim_{t\to \infty} t^{-1} \E\left[ \sum_{k=1}^{\infty} c_{1} (X(\tau_{k}-), X(\tau_{k}))I_{\{\tau_{k} \le t \}}\right] 
& =  \lim_{t\to \infty} t^{-1} \E\left[ \sum_{k=1}^{\infty} c_{1} (y,z) I_{\{\tau_{k }\le t \}}\right] \\ & = c_{1}(y,z) \kappa = \frac{c_{1}(y,z)}{B\zeta(y,z)}.
\end{aligned}$$
Combining the above two displayed equations gives \eqref{eq-sS-cost}.
\end{proof}

Motivated by Proposition \ref{prop-sS-cost} and using $g_0$ and $\zeta$ defined in \eqref{eq-g0-fn} and \eqref{eq-psi-fn} respectively, we define
\begin{equation} \label{eq-F-fn}
F(y,z) : =\frac{c_1(y,z) + Bg_0 (y,z)} {B\zeta(y,z)}, \quad (y, z) \in \mathcal R, 
\end{equation}

Under the mild conditions of this paper, there exists a minimizing pair $(y_*,z_*)$ for the function $F$.

\begin{prop} \label{F-optimizers}
Assume Conditions \ref{diff-cnd} and \ref{cost-cnds} hold.  Then there exists a pair $(y_*, z_*) \in \mathcal R$ such that
\begin{equation} \label{eq-F-optimal}
F(y_*, z_*) = F_* := \inf\set{F(y,z): (y,z) \in \mathcal R}.
\end{equation}
Furthermore, when $c_1 \in C^1({\cal R})$, an optimizing pair $(y_*,z_*)$ satisfies the following first-order conditions:
\begin{itemize}
\item[(a)] When $a$ is a natural boundary, the first-order optimality conditions imply
\begin{equation}
\label{eq-1st-op}
F_* = \frac{-\frac{\partial c_1}{\partial y} (y_*, z_*) + g_0'(y_*)}{ \zeta'(y_*)} = \frac{ \frac{\partial c_1}{\partial z} (y_*, z_*) + g_0'(z_*)}{ \zeta'(z_*)}.
\end{equation}
\item[(b)] When $a$ is an attainable boundary (regular or exit), an optimal pair $(y_*,z_*)$ may have $y_*=a$.  When $y_* \neq a$, \eqref{eq-1st-op} holds but, when $y_* = a$, the first-order conditions imply
\begin{equation} \label{eq-1st-op2}
\frac{-\frac{\partial c_1}{\partial y}(a, z_*) + g_0'(a)}{\zeta'(a)} \leq F_* = \frac{ \frac{\partial }{\partial z} c_1 (a, z_*) + g_0'(z_*)}{ \zeta'(z_*)}.
\end{equation}
\end{itemize}
\end{prop}

\begin{proof}
We examine the limiting behavior of $F$ as the pair $(y,z)$ approaches the boundaries of ${\cal R}$.

\noindent
{\em (i) The diagonal $y=z$.}\/ For each $\overline{x} \in {\cal E}$, $\zeta(\overline{x}) < \infty$ so $y, z \rightarrow \overline{x}$ implies $B\zeta(y,z) = \zeta(z) - \zeta(y) \rightarrow 0$.  By Condition \ref{cost-cnds}(b), $c_1 \geq k_1 > 0$ so
\begin{equation}
\label{eq-limit-diagonal}
\lim_{(z-y) \rightarrow 0} F(y,z) \geq \lim_{(y,z) \rightarrow (\overline x, \overline x)} \frac{c_1(y,z)}{B\zeta(y,z)} \geq \lim_{(y,z) \rightarrow (\overline x, \overline x)} \frac{k_1}{B\zeta(y,z)} = \infty.
\end{equation}

The analysis of the diagonal boundary related $F$ to the ratio $c_1/B\zeta$.  Observe that we may also obtain a lower bound on $F$ by
\begin{equation} \label{F-lb1}
F(y,z) \geq \frac{Bg_0(y,z)}{B\zeta(y,z)} = \frac{g_0(z) - g_0(y)}{\zeta(z) - \zeta(y)}.
\end{equation}
This bound will be utilized in the analysis of the other boundaries.

\noindent
{\em (ii) The upper boundary $z=b$ with $y \in {\mathcal I}$ fixed.}\/
Recalling the definition of $\zeta $ in \eqref{eq-psi-fn}, we have 
\begin{equation}
\label{eq-B-psi=N(b)}
\mbox{$\frac{1}{2}$} B\zeta (y,z) = \int_{y}^{z} M[u,b) \d S(u) \to \int_{y}^{b }M[u,b) \d S(u) =: N(b), \quad \text{ as }z \nearrow b,
\end{equation}
in which the function $N$ is defined by equation (7.2) on page 242 of \cite{KarlinT81}.

When $N(b) < \infty$, with reference to Table 6.2 on page 234 (and also Table 7.1 on page 250) of  \cite{KarlinT81},  $b$ is an entrance boundary point.  In this case,  $\lim_{z\nearrow b} B\zeta (y,z) < \infty$.  But \eqref{infinite-dbl-intgrl-at-b} of Condition \ref{cost-cnds}(a)  implies that $\lim_{z\nearrow b} Bg_{0}(y, z) = \infty$. Therefore we have $\lim_{z\nearrow b} F(y,z) = \infty$.

 When $N(b) = \infty$, Table 6.2 of \cite{KarlinT81} shows that $b$ is a natural boundary and $\lim_{z\nearrow b} B\zeta (y,z) =\infty$.  Again, we have  $\lim_{z\nearrow b} Bg_{0}(y, z) = \infty$ by  \eqref{infinite-dbl-intgrl-at-b}. Then an applications of   L'H\^{o}pital's rule with the definitions of $g_0$ and $\zeta$ in \eqref{eq-g0-fn} and \eqref{eq-psi-fn} yields
$$\lim_{z \nearrow b} F(y,z) \ge \lim_{z \nearrow b} \frac{g_0(z) - g_0(y)}{\zeta(z) - \zeta(y)} = \lim_{z\nearrow b} \frac{\int_z^b c_0(v)\, \d M(v)}{\int_z^b 1\, \d M(v)}.$$
Since $c_0$ is inf-compact and $b \notin {\cal E}$, for each $K > 0$, there exists some $z_K \in {\cal I}$ such that $c_0(z) \geq K$ for all $z \geq z_K$.  Hence for all $z \geq z_K$,
$$\frac{\int_z^b c_0(v)\, \d M(v)}{\int_z^b 1\, \d M(v)} \geq \frac{\int_z^b K\, \d M(v)}{\int_z^b 1\, \d M(v)} = K.$$
Since $K$ is arbitrary, it follows that\begin{equation}
\label{eq-limit-z-to-b}
 \lim_{z\nearrow b} F(y,z) \ge \lim_{z\nearrow b}  \frac{g_0(z) - g_0(y)}{\zeta(z) - \zeta(y)} =\infty.
\end{equation}
Thus \eqref{eq-limit-z-to-b} is valid both when $b$ is a natural boundary and when it is an entrance boundary.

\noindent
{\em (iii) The vertex $(b,b)$.}\/ When $(y,z) \rightarrow (b,b)$  and $b$ is a natural boundary, for any $K> 0$, taking $y > z_K$ directly (with $z_K$ specified above) gives
$$F(y,z) \geq  \frac{g_0(z) - g_0(y)}{\zeta(z) - \zeta(y)} = \frac{\int_y^z \int_u^b c_0(v) \d M(v) \d S(u)}{\int_y^z \int_u^b 1 \d M(v) \d S(u)} \geq K.$$ When $b$ is an entrance boundary, \eqref{eq-limit-diagonal} gives $\lim_{(y,z)\to (b,b)} F(y,z)= \infty$.

It now remains to analyze the asymptotic behavior of $F(y,z)$ on the boundary $y=a$ with $z\in {\cal I}$ fixed, and at the two vertices $(a,a)$ and $(a,b)$.  We begin by determining the asymptotic behavior of $B\zeta(y,z)$ as $y\searrow a$ with $z \in {\cal I}$ fixed.  Note that for any $(y,z) \in \mathcal R$ with $z\in \I$ fixed, by the definition of $\zeta $ in \eqref{eq-psi-fn},
\begin{equation}\label{fn-Sigma-B-psi}
\begin{aligned}
 \mbox{$\frac{1}{2}$} B\zeta(y,z) & =  \int_{y}^{z} M[u,b) \d S(u)
      = \int_{y}^{z} M[u,z] \d S(u) + M[z,b) S[y,z] \\
    &  \stackrel{y \downarrow a}{\longrightarrow} \int_{a}^{z} M[u,z] \d S(u) + M[z,b) S(a,z]\\
    &=:\Sigma(a)  +  M[z,b) S(a,z], 
\end{aligned}
\end{equation}
in which the function $\Sigma$ is defined by equation (6.10) on page 229 of \cite{KarlinT81}.  Moreover by Table~6.2 on page 234 of \cite{KarlinT81}, the point $a$ is attainable or unattainable according to whether $\Sigma(a) < \infty$ or $\Sigma(a) = \infty$, respectively.  By Proposition \ref{M-finite}, $M[z,b) < \infty$ for any $z\in \I$ fixed. Also, for any $z\in \I$ fixed, we have $S(a,z] < \infty$ by Condition \ref{diff-cnd}(a). Thus it follows from \eqref{fn-Sigma-B-psi} that
\begin{eqnarray} \label{psi-asymptotics-at-a1}
\lim_{y\rightarrow a} B\zeta(y,z) &<& \infty, \qquad \mbox{when $a$ is attainable, and } \\ \label{psi-asymptotics-at-a2}
\lim_{y\rightarrow a} B\zeta(y,z) &=& \infty, \qquad \mbox{when $a$ is unattainable.}
\end{eqnarray}
We now examine the remaining cases for the asymptotic behavior of $F(y,z)$.

\noindent
{\em (iv) The left boundary $y=a$ with $z\in {\cal I}$ fixed.}\/  Fix $z \in {\cal I}$.  The asymptotic behavior of $F(y,z)$ at this boundary depends on the type of boundary point for $a$. 

\noindent $\bullet$
  When $a$ is unattainable, 
	a simple comparison between $\zeta$ and $g_0$ in \eqref{eq-g0-fn} and \eqref{eq-psi-fn} establishes that $\lim_{y \rightarrow a} Bg_0(y,z) = \infty$.  Due to \eqref{psi-asymptotics-at-a2}, we can apply L'H\^{o}pital's rule to obtain
\begin{equation}
\label{eq-limit-ya-F(y,z)}
 \lim_{y\downarrow a} F(y,z) \ge \lim_{y\downarrow a} \frac{g_0(z) - g_0(y)}{\zeta(z)-\zeta(y)} = \lim_{y\downarrow a} \frac{\int_y^b c_0(v) \d M(v)}{\int_y^b 1 \d M(v)} =\lim_{y\downarrow a}  \frac{\int_y^b c_0(v) \d M(v)}{M[y,b)}.
\end{equation}
Furthermore, By Lemma 6.3(v) on page 231 of \cite{KarlinT81}, under the conditions that $S(a, x] < \infty$ for any $x\in \I$ and $\Sigma(a) = \infty$, we must have $M(a, x] = \infty$ for any $x\in \I$. Thus we have $\lim_{y\downarrow a} M[y,b) = \infty$. This limit, along with the estimate in \eqref{eq-M-c0-integral}, establishes that $\lim_{y\downarrow a} \int_y^b c_0(v) \d M(v)  = \infty$. Thus we can again apply  L'H\^{o}pital's rule to the rightmost expression in \eqref{eq-limit-ya-F(y,z)} to obtain
\begin{equation}
\label{eq-limit-F-y-to-a}
 \lim_{y\downarrow a} F(y,z) \ge \lim_{y\downarrow a} \frac{g_0(z) - g_0(y)}{\zeta(z)-\zeta(y)} \ge \lim_{y\downarrow a}c_{0} (y) = \infty,
\end{equation} where the last equality follows from the inf-compactness assumption for the function $c_{0}$ in Condition \ref{cost-cnds}.

 \noindent $\bullet$
   When $a$ is an attainable boundary, \eqref{psi-asymptotics-at-a1} implies that one may pass to the limit to obtain
$$\lim_{y\rightarrow a} F(y,z) = \frac{c_1(a,z) + g_0(z) - g_0(a)}{\zeta(z) - \zeta(a)}.$$
It is possible that $y_*=a$ if $g_0(a) > -\infty$.  This possibility is accounted for in the optimization \eqref{eq-F-optimal} since $a \in {\cal E}$.

 \noindent
{\em (v) The vertex $(a,b)$.}\/  We need to consider separately the asymptotics  when $a$ is attainable and  when $a$ is unattainable.

\noindent $\bullet$  Suppose $a$ is unattainable. We fix some $y_{0}\in \I$ and write
\begin{equation}\label{eq-F>g0/psi}
F(y,z)  \ge \frac{g_{0}(z) - g_{0}(y)}{\zeta(z) - \zeta(y) } = \frac{g_{0}(z) - g_{0}(y_{0}) +g_{0}(y_{0})  -g_{0}(y) }{\zeta(z) -\zeta(y_{0})  + \zeta(y_{0}) -\zeta(y) }.
\end{equation}
By \eqref{eq-limit-F-y-to-a}, for any $K > 0$ we can pick some $y_{K}\in \I$ so that
\begin{equation}
\label{eq-g0>psi-y-a}
  g_{0}(y_{0})  -g_{0}(y) > K (\zeta(y_{0}) -\zeta(y)) \text{ for all }a < y < y_{K}.
\end{equation}
Similarly, \eqref{eq-limit-z-to-b} implies that we can pick some $z_{K} \in \I$ so that
\begin{equation}
\label{eq-g0>psi-z-b}
 g_{0}(z) - g_{0}(y_{0}) > K (\zeta(z) - \zeta(y_{0})) \text{ for all  } z_{K }< z < b.
\end{equation}
Putting \eqref{eq-g0>psi-y-a} and \eqref{eq-g0>psi-z-b} into \eqref{eq-F>g0/psi} yields that $F(y,z)  > K$ for all $a < y < y_{K}$  and $z_{K }< z < b$. Since $K>0$ is arbitrary, this shows that $$\liminf_{(y,z)\rightarrow (a,b)} F(y,z)  = \infty.$$

\noindent $\bullet$ Suppose $a$ is attainable.   Arbritrarily fix $y_0 \in {\cal I}$.  Due to Condition \ref{diff-cnd}(a), Proposition \ref{M-finite},  \eqref{eq-B-psi=N(b)}, and \eqref{fn-Sigma-B-psi}, we have for any $a \le y < z < b$
\begin{equation} \label{psi-est-5}
\begin{aligned} 
\mbox{$\frac{1}{2}$}B\zeta(y,z) & = \mbox{$\frac{1}{2}$}[\zeta(z) - \zeta(y_{0})] +  \mbox{$\frac{1}{2}$}[\zeta(y_{0}) - \zeta(y) ] \\
& \le N(b) + \Sigma(a) + M[y_{0}, b) S(a, y_{0}] < \infty.
\end{aligned}
\end{equation}
We now consider the two cases when $b$ is an entrance or a natural boundary separately.
\begin{itemize}
  \item[$\circ$] When $b$ is an entrance boundary, $N(b) < \infty$.  Assumption \eqref{infinite-dbl-intgrl-at-b} of Condition \ref{cost-cnds} implies that for any $K> 0$, we can select some $z_{K} \in \I$ so that for all $b> z > z_{K}$
  $$ \mbox{$\frac{1}{2}$}Bg_{0}(y,z) \ge \mbox{$\frac{1}{2}$} [g_{0}(z) - g_{0}(y_{0}) ] > K [N(b) + \Sigma(a) + M[y_{0}, b) S(a, y_{0}]]$$
and as a result
$$ \lim_{(y,z) \to (a,b)} F(y,z) \ge \lim_{(y,z) \to (a,b)}  \frac{Bg_{0}(y,z)}{B\zeta(y,z) } = \infty.$$

\item[$\circ$] When $b$ is a natural boundary, $N(b) = \infty$.  By \eqref{eq-limit-z-to-b} in case (ii), we can pick some $z_1\in \I$ such that for all $z_{1} < z < b$,
\begin{displaymath}
\frac{g_{0}(z) - g_{0}(y_{0})}{ \zeta(z) - \zeta(y_{0}) } > 2 K.
\end{displaymath}
By \eqref{psi-est-5} and \eqref{eq-B-psi=N(b)}, we can pick some $z_2\in \I$ such that for all $z_2 < z < b$ and  $a \le y < y_0$,
$$ \frac{\zeta(y_{0}) - \zeta(y) }{\zeta(z) - \zeta(y_{0}) }  < 1. $$
Set $z_{K} = z_1 \vee z_2$. Then we have for all $z_{K}< z < b$ and $a \le y < y_{0}$,
\begin{displaymath}\begin{aligned}
F(y,z) &  \ge \frac{Bg_{0}(y,z)}{B\zeta(y,z)} \ge \frac{g_{0}(z) - g_{0}(y_{0})}{ \zeta(z) - \zeta(y_{0}) + \zeta(y_{0}) - \zeta(y)} \\
          & \ge  \frac{ 2K[ \zeta(z) - \zeta(y_{0})  ]}{ [\zeta(z) - \zeta(y_{0}) ] +  [\zeta(z) - \zeta(y_{0}) ] } = K.
\end{aligned}\end{displaymath} Since $K>0$ is arbitrary, we conclude that $\liminf_{(y,z)\rightarrow (a,b)} F(y,z)  = \infty.$
\end{itemize}

\noindent
{\em (vi) The vertex $(a,a)$.}\/  When $a$ is attainable, $\lim_{x\rightarrow a} \zeta(x) > -\infty$ so $F(y,z)\rightarrow \infty$ as $(y,z) \rightarrow (a,a)$ by case (i).

Now assume $a$ is unattainable so $a$ is a natural boundary. We fix some $y_{0}\in\I$. In view of \eqref{eq-limit-F-y-to-a}, for any $K> 0$, we can pick some $z_{K} \in \I$ so that for all $a < y < z < z_{K}$, we have\begin{displaymath}
\frac{g_{0}(z) - g_{0}(y_{0}) } {\zeta(z) - \zeta(y_{0}) } > K, \text{ and } \frac{g_{0}(y_{0}) - g_{0}(y) } {\zeta(y_{0}) - \zeta(y) } > K.
\end{displaymath}
Thus it follows that \begin{displaymath}\begin{aligned}
F(y,z) &  \ge \frac{Bg_{0}(y,z)}{B\zeta(y,z)}  =  \frac{g_{0}(z) - g_{0}(y_{0})+g_{0}(y_{0}) - g_{0}(y) }{ \zeta(z) - \zeta(y_{0}) + \zeta(y_{0}) - \zeta(y)}  \\
  & > \frac{ K[ \zeta(z) - \zeta(y_{0})] + K [\zeta(y_{0}) - \zeta(y)]}{ \zeta(z) - \zeta(y_{0}) + \zeta(y_{0}) - \zeta(y)} = K.
\end{aligned}\end{displaymath}
Again, since $K$ is arbitrary, it follows that $\lim_{(y,z)\rightarrow (a,a)} F(y,z) =\infty$.

As a result of this asymptotic behavior of $F$ and $F$ being a continuous function, the minimum of $F$ is achieved at some point $(y_*,z_*) \in {\cal R}$.  Note that when $y_* \neq a$, the first-order conditions on $F$ are satisfied at $(y_*,z_*)$ and these conditions can be rearranged to yield \eqref{eq-1st-op}.  For those models for which $y_* = a$, $\frac{\partial F}{\partial y}(a,z_*) \geq 0$ and this condition yields \eqref{eq-1st-op2}.
\end{proof}

We now give a technical result that will be used in the proof of Proposition \ref{prop-G-qvi}.
\begin{prop}\label{lem-h-fn}
Assume Conditions \ref{diff-cnd} and \ref{cost-cnds} hold. Let $(y_*,z_*)$ be an optimizing pair of $F$.  Suppose $y_{*} \in {\cal I}$, $c_1 \in C^1({\cal R})$ and $c_1(\cdot,z_*) \in C^2({\cal I})$.  Then 
\begin{equation}
\label{ineq-h-fn}
\mbox{$\frac{\partial^{2} c_{1}}{\partial y^{2}}$}(y_{*},z_{*}) - g_{0}''(y_{*})  + F_{*} \zeta''(y_{*}) \ge 0.
\end{equation}
\end{prop}

\begin{proof}
Since the function $F$ achieves its minimum value at $(y_{*},z_{*})$ by Proposition \ref{F-optimizers} and $y_{*} \in (a, b)$ by assumption, we have $\frac{\partial}{\partial y} F(y_{*},z_{*})= 0 $ and $\frac{\partial ^{2}}{\partial y^{2}} F(y_{*},z_{*}) \ge 0$.  The inequality leads to \eqref{ineq-h-fn} by straightforward but tedious calculations utilizing both \eqref{eq-F-optimal} and  \eqref{eq-1st-op}.
\end{proof}

\subsection{Optimality in a large class of policies} \label{subsect:gen-soln}
We now identify a large class of admissible ordering policies in which an optimal $(s,S)$ policy remains optimal.

Define for all $x \in \I$
\begin{equation} \label{eq-G-fn}
G(x) = \begin{cases}
 c_1(x,z_*) + g_0 (z_*) - F_* \zeta (z_*),     & \text{ if } x \le y_*, \\
 g_0 (x) - F_* \zeta (x),    & \text{ if  } x >   y_*. 
\end{cases}
\end{equation}

\begin{rem}\label{rem-about-G}
We make several observations concerning the function $G$ here.

When $b$ is an entrance boundary, in view of \eqref{eq-B-psi=N(b)},  $G$ is   bounded below on $[y_{*}, b)$.
 When $x \ge  y_{*}$, thanks to the definitions of $\zeta $ and $g_{0}$ in \eqref{eq-psi-fn} and \eqref{eq-g0-fn}, respectively, we have
$$ \begin{aligned}
G'(x) = g_{0}'(x) - F_{*}\zeta'(x) = 2 s(x)  \int_{x}^{b} [c_{0}(v) - F_{*}] m(v)\d v.
\end{aligned}$$ Recall from Condition \ref{cost-cnds} (a) that $c_{0}$ is inf-compact. Thus when $b$ is a natural boundary, there exists some $z_{F_*}$ with $y_{*} < z_{F_*} < b$ so that  $c_{0}(x) > F_{*}$ for all $x \in [z_{F_*}, b)$. This shows that $G'(x) > 0$ for $x \ge z_{F_*}$; that is, $G$ is strictly increasing  on $[z_{F_*}, b)$ and therefore bounded below on $[y_{*}, b)$.

Now, from the monotonicity of $c_1$ in \eqref{eq-c1-monotone}, when $x_{1} < x_{2} < y_{*}$, we have that $c_{1}(x_{1}, z_{*}) \ge c_{1}(x_{2},z_{*})$ and hence $G(x_{1}) \ge G(x_{2})$.  Thus $G$ is decreasing on $(a, y_{*})$ and therefore $G$ is bounded below on  $(a, y_*]$ as well.

Therefore it follows that there exists some positive constant  $\widehat{\kappa}$ so that 
\begin{equation} \label{eq-G-bdd-above}
G(x) \ge -\widehat{\kappa}, \quad \text{ for all }x \in \I.
\end{equation}
\end{rem}

The next proposition requires a technical condition in the event that the optimizer $y_*$ is an interior point of ${\cal I}$.

\begin{cnd} \label{eq-AG-c0}
If $y_{*} > a$, then $c_1 \in C^1({\cal R})$, $c_1(\cdot,z_*) \in C^2({\cal I})$ and the function $A G + c_{0}$ is decreasing on $(a,y_{*})$.
\end{cnd}

\begin{prop}\label{prop-G-qvi} Assume Conditions \ref{diff-cnd}, \ref{cost-cnds} and \ref{eq-AG-c0} hold.
Then the function $G$ defined by  \eqref{eq-G-fn} is in $C^{1}(\I)\cap C^{2}(\I - \set{y_{*}})$ and satisfies the system of constraints:
\begin{equation}
\label{eq-G-qvis}
\begin{cases}
    AG(x) + c_0(x) - F_* \ge  0,  & \text{ for all  }\ \ x\in \I - \set{y_{*}}, \\
    BG(y,z)   \ge -c_1 (y,z),  &  \text{ for all  }\ \  (y,  z) \in \mathcal R, \\
    AG(x) + c_0(x) - F_* =  0,  &  \text{ for all  }\ \ x >   y_*,  \\
    BG(x,z_*) = -c_1  (x,z_*),  &  \text{ for all  }\ \ x \le  y_*.
    \end{cases}
\end{equation}
\end{prop}

\begin{rem} \label{rem-qvi}
Considering the entire system \eqref{eq-G-qvis}, the function $G$ and the constant $F_*$ provide a solution to the QVI:
\begin{equation} \label{qvi}
 \min\left\{\,AG(x) + c_0(x) - F_*,\, \min_{z\in {\cal I}} [BG(x,z) + c_1(x,z)]\,\right\} = 0, \quad \forall x \in {\cal I}.
\end{equation}
\end{rem}

\begin{proof} 
For those models in which $y_* = a$, the first condition of \eqref{eq-G-qvis} is the same as the third condition and the last condition is trivially true.  When $y_* = a$, the proof below establishes the third condition of \eqref{eq-G-qvis} in \eqref{eq1b-G-equality} while the second condition of \eqref{eq-G-qvis} is verified by Case (i) of \eqref{eq1-G-qvi}.  We therefore prove \propref{prop-G-qvi} when $(y_*,z_*)$ is an interior point of ${\cal R}$, establishing the result in general.

To see that $G\in C^{1}(\I)\cap C^{2}(\I - \set{y_{*}})$, it is enough to examine that $G$ is $C^{1} $ at $y_{*}$. In fact, we have
\begin{displaymath}
\begin{aligned}
  G(y_*-)  & = c_1(y_{*},z_*) + g_0 (z_*) - F_* \zeta (z_*)  \\
    &  = g_0(z_*) - F_* \zeta (z_*)  + \( F_* [\zeta(z_*) -\zeta(y_*)] - [g_0(z_*)- g_0(y_*)]\) \\
    &= g_0(y_*) - F_* \zeta(y_*) = G(y_*+);
\end{aligned}
\end{displaymath}
and from \eqref{eq-1st-op}, \begin{displaymath}
G'(y_*+) = g_0'(y_*) - F_* \zeta'(y_*) = \mbox{$\frac{\partial c_{1}}{\partial y}$}(y_{*},z_{*}) =  G'(y_*-).
\end{displaymath}

 The rest of the proof is to show that each equality and inequality of \eqref{eq-G-qvis} is satisfied.  
 First we observe that by the definition of $G(x)$,
\begin{align}
\label{eq1a-G-equality}
& BG(x,z_*)  = -c_1(x, z_*), & \text{ for all } x \le y_*,
\\ \label{eq1b-G-equality} &  AG(x) + c_0(x) - F_*  =  0,  &  \text{ for all  }  x >   y_*.
\end{align}
Next we show that     \begin{equation}
\label{eq1-G-qvi}
BG(y,z) \ge -c_1(y,z) \text{ for all }  (y,  z) \in \mathcal R
\end{equation}
 in three cases.
\begin{description}
  \item[Case (i)] $ z > y \ge y_*$.  We have from the definition of $F_{*}$ that
  \begin{displaymath}
\begin{aligned}
 BG(y, z)  &  = [g_0(z) - g_0(y)] - F_* [\zeta(z) - \zeta(y)] \\
    & \ge [g_0(z) - g_0(y)] - \frac{c_1  (y,z) + g_0(z) - g_0(y) }{ \zeta(z) - \zeta(y) }[\zeta(z) - \zeta(y)] \\
    & = -c_1(y,z).
\end{aligned}
\end{displaymath}
  \item[Case (ii)] $y < z \le y_*$. \begin{displaymath}
\begin{aligned}
     BG(y,z)  & =   [c_1(z, z_{*}) + G(z_*)] -  [c_1 (y,z_{*}) + G(z_*)] \\  
		 &= c_1 (z, z_{*})  - c_1 (y, z_{*} ) \\
     & \ge -c_{1}(y,z),
    \end{aligned}
\end{displaymath} where the last inequality follows from \eqref{eq-c1-sub-add}.
  \item[Case (iii)] $y < y_* \le z $. Using the results established in the previous two cases, we have
  \begin{displaymath}
\begin{aligned}
   BG(y,z)  &     = (G(z) - G(y_*)) + (G(y_*) - G(y)) \\
     &  \ge  -c_1 (y_*,z) +  c_1(y_*,z_*) - c_1( y,z_{*}) \\
     & = -c_{1} (y,z),
\end{aligned}
\end{displaymath} 
in which the inequality above follows from Case (i) with $y= y_{*}$ for the first difference and, with $z=y_*$, from the second equality in Case (ii) for the second difference.  The last line is a consequence of \eqref{eq-c1-equal-displacement}. 
\end{description}

To finish the proof, it remains to show that
\begin{equation}
\label{eq2-G-qvi}
A G(x) + c_0(x) - F_*   \ge 0 ,\ \text{ for all }x \in \I - \{y_*\}.
\end{equation}
If $x > y_*$, we have \eqref{eq2-G-qvi} holds with equality by \eqref{eq1b-G-equality}.  We must therefore show  \eqref{eq2-G-qvi}  for $ x < y_*$.  To prove \eqref{eq2-G-qvi} for $x\in (a, y_{*})$, it is enough to show that $A G(y_*-) + c_{0}(y_*-) + F_{*} \ge 0$ since the function $AG + c_{0}$ is decreasing on $(a, y_{*})$ by \cndref{eq-AG-c0}. To this end, we use the facts that the functions  $\mu,\sigma$ and $c_{0}$ are continuous and that $G \in C^{1}(\I)$ to compute
\begin{align*}
 A  &G(y_*-) + c_{0}(y_*-) + F_{*}  \\  
    & = \mu (y_*-) G'(y_*-) + \mbox{$\frac{1}{2}$} \sigma^{2}(y_*-)G'' (y_*-) + c_{0}(y_*-)  + F_{*}  \\
    &  = \left[ \mu (y_{*}) G'(y_{*}) + \mbox{$\frac{1}{2}$} \sigma^{2}(y_{*})G'' (y_*+) + c_{0}(y_{*})  + F_{*} \right] + \mbox{$\frac{1}{2}$} \sigma^{2}(y_{*}) \left[  G'' (y_*-) - G'' (y_*+) \right] \\
    & = 0 + \mbox{$\frac{1}{2}$} \sigma^{2}(y_{*}) \left[\mbox{$\frac{\partial^{2} c_{1}}{\partial y^{2}}$}(y_{*},z_{*}) - g_{0}''(y_{*})  + F_{*} \zeta''(y_{*})  \right]\\
    & \ge 0,
\end{align*}
where we used the equality in \eqref{eq2-G-qvi} for $x= y_{*}$ to obtain the third equality, while the last inequality follows from  \eqref{ineq-h-fn}. The proof is now complete.
\end{proof}

As a result of \propref{prop-G-qvi}, \remref{rem-qvi} observes that the function $G$ defined by \eqref{eq-G-fn} and the constant $F_*$ satisfy the QVI \eqref{qvi}.  However, care must be taken to utilize the system \eqref{eq-G-qvis} to establish optimality of the policy \eqref{sS-tau-def} in which $(y,z) = (y_*,z_*)$.  For an ordering policy $(\tau,Y)$, let $X$ denote the resulting inventory level process and define the process $\widehat{M}$ by
\begin{equation} \label{eq-M-process}
\widehat{M}(t):= \int_{0}^{t} G'(X(s))\sigma (X(s)) \d W(s), \qquad t \geq 0.
\end{equation}
At this level of generality, it is not immediate that $\widehat{M}$ is a martingale nor that $\E[G(X(t))]/t$ vanishes as $t\rightarrow \infty$.  Thus we prove optimality in a subclass ${\cal A}_1$, which we now define.  Recall by \eqref{eq-G-bdd-above}, $G$ is bounded below which prevents the limit below (in $t$) from being negative.

\begin{defn} \label{A1-def}
Let ${\cal A}_1 \subset {\cal A}$ consist of those ordering policies $(\tau,Y)$ for which 
\begin{description}
\item[(i)] there exists an increasing sequence of stopping times $\{\beta_n: n \in \N\}$ with $\beta_n \rightarrow \infty\; (a.s.)$ such that $\{ \widehat{M}(t\wedge \beta_{n}): 0 \le t < \infty\}$ is a martingale for every $n\in \N$, 
\item[(ii)] the transversality condition 
\begin{equation}
\label{eq-transversality}
\liminf_{t\to \infty} \liminf_{n\to \infty}t^{-1} \E[G(X(t\wedge \beta_{n}))] = 0
\end{equation}
holds, and
\item[(iii)] when $a$ is a reflective boundary, the transversality condition
\begin{equation} \label{trans-local}
\lim_{t\rightarrow \infty} t^{-1} \E[L_a(t)] = 0
\end{equation} in which $L_a$ denotes the local time process of $X$ at $\{a\}$.
\end{description}
\end{defn}

\begin{rem}\label{rem-M-mg-transversality}
If the process $\{\widehat{M}(t): 0 \le t < \infty\}$ defined in \eqref{eq-M-process} is already a martingale, then for each $n$, $\beta_n$ may be taken to be $\infty$ and the transversality condition \eqref{eq-transversality} reduces to
\begin{equation} \label{eq-transversality2}
\liminf_{t\rightarrow \infty} t^{-1} \E[G(X(t))] = 0.
\end{equation}
\end{rem}

We begin by showing that the class of $(s,S)$ policies is in ${\cal A}_1$ so it is non-empty.  Recall that the set of $(y,b)$ policies when $b$ is an entrance boundary incur infinite cost so we exclude $z=b$ in the following analysis.

\begin{prop} \label{prop:sS-in-A1}
Let $X(0-) = x_0$.  Fix $(y,z) \in {\cal R}$ with $z < b$ and define the ordering policy $(\tau,Y)$  by \eqref{sS-tau-def}.  Then $(\tau,Y) \in {\cal A}_1$.
\end{prop}

\begin{proof}
Let $(\tau,Y)$ be as in the statement of the proposition  and let $X$ denote the resulting inventory process satisfying \eqref{controlled-dyn}.  Note that in the case that the boundary point $a$ is reflective, regardless of the value of $y$, an order will occur before any reflection should the process ever be at $a$.

Let $\{b_{n}: n\in \N\}$ be a strictly increasing sequence with $x_0 \vee z < b_n < b$ for each $n$ and $\lim_{n\to\infty} b_{n} = b$.  Define $\hat \beta_{n}: =\inf\{t\ge 0: X(t) = b_{n}\}$. Then $\hat\beta_{n}$ is an increasing sequence of stopping times and since $b$ is a non-attracting point, we have $\lim_{n\to\infty} \hat\beta_{n} =\infty\; (a.s.)$. Moreover, since $y \leq X(t\wedge \hat{\beta}_n) \leq b_n \;(a.s.)$ for all $t \ge 0$, the process $\widehat{M}(\cdot\wedge\hat\beta_{n})$ is a martingale.  

Recall from \eqref{eq-G-bdd-above} that there exists some $\widehat{\kappa}$ such that $G \geq -\widehat{\kappa}$. We now identify a function which bounds $G$ above on the interval $[y,b)$.  When $y \geq y_*$, then $G(y) = g_0(y) - F_* \zeta(y)$.  When $y < y_*$, the continuity of $G$ and $g_0 - F_* \zeta$ implies that these functions are bounded on $[y,y_*]$ and hence there exists some $\widetilde{\kappa} > 0$ such that $G \leq \widetilde{\kappa} + g_0 - F_* \zeta$ on $[y,y_*]$.  As a result,
\begin{equation} \label{G-bounds}
-\widehat{\kappa} \leq G(X(t)) \leq \widetilde{\kappa} + g_0(X(t)) - F_* \zeta(X(t)) \quad \forall t \geq 0
\end{equation}
and the transversality condition \eqref{eq-transversality} follows provided 
$$\lim_{t\rightarrow \infty} \lim_{n\rightarrow \infty} t^{-1} \E[g_0(X(t\wedge \hat\beta_n)) - F_* \zeta(X(t\wedge \hat\beta_n))] = 0.$$  

By It\^{o}'s formula using the function $g_0 - F_* \zeta$, we have
\begin{eqnarray*}
[g_0 - F_* \zeta](X(t\wedge\hat\beta_n)) &=& [g_0 -F_*\zeta](x_0) + \int_0^{t\wedge\hat\beta_n} A[g_0-F_*\zeta](s)\, \d s \\
& & + \sum_{k=1}^\infty I_{\{\tau_k \leq t\wedge\hat\beta_n\}} B[g_0-F_*\zeta](X(\tau_k-),X(\tau_k)) \\
& & + \int_0^{t\wedge \hat\beta_n} \sigma(X(s)) [g_0-F_*\zeta]'(X(s))\, \d W(s).
\end{eqnarray*}
Observe that the integrand in the stochastic integral is bounded over the range of integration, $X(\tau_k-) = y$ and $X(\tau_k) = z$, and that $Ag_0 = - c_0$ and $A\zeta = -1$.  Thus by the optional sampling theorem, taking expectations results in
\begin{align*}
t^{-1} \E\left[[g_0 - F_* \zeta](X(t\wedge\hat\beta_n))\right] &=  \frac{[g_0-F_*\zeta](x_0)}{t} - t^{-1} \E\left[\int_0^{t\wedge \hat\beta_n} c_0(X(s))\, \d s\right] \\
&  \quad\ + \frac{F_* \E[t\wedge\hat\beta_n]}{t} + \frac{B[g_0-F_*\zeta](y,z)}{t} \E\left[\sum_{k=1}^\infty I_{\{\tau_k \leq t\wedge\hat\beta_n\}}\right].
\end{align*}
By an application of the monotone convergence theorem, we have 
\begin{align} \label{g0-Fzeta-id} \nonumber
  \lim_{n\rightarrow \infty}  \frac{1}{t } \E\left[[g_0 - F_* \zeta](X(t\wedge\hat\beta_n))\right]    &= \frac{[g_0-F_*\zeta](x_0)}{t} - \frac{1}{t } \E\left[\int_0^{t} c_0(X(s))\, \d s\right]  +   F_* \\ & \quad \
  + \frac{B[g_0-F_*\zeta](y,z)}{t} \E\left[\sum_{k=1}^\infty I_{\{\tau_k \leq t\}}\right] 
\end{align}
since $\hat\beta_n \rightarrow \infty\; (a.s.)$. Now letting $t\rightarrow \infty$, \eqref{lta-running-cost} yields
$$t^{-1}\E\left[\int_0^t c_0(X(s))\, \d s\right] = \frac{Bg_0(y,z)}{B\zeta(y,z)}$$
and by Proposition \ref{sS-stationary}, the long-term expected frequency of orders is
$$\lim_{t\rightarrow \infty} t^{-1} \E\left[\sum_{k=1}^\infty I_{\{\tau_k \leq t\}}\right] = \frac{1}{B\zeta(y,z)}.$$
Therefore the right-hand side of \eqref{g0-Fzeta-id} converges to $0$ as $t\rightarrow \infty$.  Using the bounds \eqref{G-bounds}, we have
$$0 = \lim_{t\rightarrow \infty} \frac{-\kappa}{t} \leq \lim_{t\rightarrow \infty} \lim_{n\rightarrow \infty} \frac{\E[G(X(t\wedge\hat\beta_n))]}{t} \leq \lim_{t\rightarrow \infty}\lim_{n\rightarrow \infty} \frac{\wdt \kappa+ [g_0-F_*\zeta](X(t\wedge\hat\beta_n))}{t} = 0,$$
establishing \eqref{eq-transversality}.
\end{proof}

The following proposition identifies additional policies in ${\cal A}_1$.  The hypotheses of this proposition can be verified for many policies in certain models.

\begin{prop} \label{add-pol}
Let $(\tau,Y) \in {\cal A}$ with $J(\tau,Y) < \infty$ and let $X$ denote the resulting inventory process.  Suppose that $\widehat{M}$ defined by \eqref{eq-M-process} is a martingale, and there exist constants $C_1, C_2 > 0$ such that $\E[G(X(t))] \leq C_1 \E[c_0(X(t))] + C_2$ for all $t \geq 0$.  Then $(\tau,Y) \in {\cal A}_1$.
\end{prop}

\begin{proof}
Let $(\tau,Y)$, $X$ and $C_1, C_2$ be as in the statement of the proposition.  Since $\widehat{M}$ is a martingale, \eqref{eq-transversality} reduces to \eqref{eq-transversality2}.  

Now assume that \eqref{eq-transversality2} fails so there exists some $\delta > 0$ such that 
$$\liminf_{t\rightarrow \infty} t^{-1} \E[G(X(t))] \geq 2\delta.$$  
Then there exists some $T < \infty$ such that $\E[G(X(t))] \geq \delta\cdot t$ for all $t \geq T$.  Using the relation between $G$ and $c_0$ in the hypothesis, it follows that 
\begin{eqnarray*}
\infty &>& J(\tau,Y) = \limsup_{t\rightarrow \infty} t^{-1} \E\left[\int_0^t c_0(X(s))\, \d s + \sum_{j=1}^\infty I_{\{\tau_j \leq t\}} c_1(X(\tau_j-),X(\tau_j))\right] \\
&\geq& \limsup_{t\rightarrow \infty} t^{-1} \int_T^t \left(\mbox{$\frac{1}{C_1} \E[G(X(s))] - \frac{C_2}{C_1}$}\right)\, \d s \\
&\geq& \limsup_{t\rightarrow \infty} t^{-1} \int_T^t \left(\mbox{$\frac{\delta}{C_1}\cdot s - \frac{C_2}{C_1}$}\right)\, \d s = \infty.
\end{eqnarray*} 
This contradiction implies that \eqref{eq-transversality2} holds and hence that $(\tau,Y) \in {\cal A}_1$.
\end{proof}

We now give a sufficient condition for $\widehat{M}$ to be a martingale.

\begin{lem} \label{lem:mg}
Suppose $(\tau,Y) \in {\cal A}$ has resulting process $X$ such that all post-order locations $\{X(\tau_j): j\in \N\}$ are uniformly bounded above.  Assume the model is such that 
\begin{description}
\item[(a)] $\frac{\partial c_1}{\partial y}(x,z_*) \sigma(x)$ is bounded on $(a,y_*)$; 
\item[(b)] there exists some $\overline{y} < b$ such that $|G'(x) \sigma(x)|$ is non-decreasing on $(\overline{y},b)$; and
\item[(c)] for each $(y,z) \in {\cal R}$, 
$$\int (G'(x) \sigma(x))^2 \pi(x)\, \d x < \infty,$$
in which $\pi$ denotes the stationary density of \eqref{eq-pi-defn} for the $(y,z)$-ordering policy.
\end{description} 
Then $\widehat{M}$ is a martingale.
\end{lem}

\begin{proof}
Let $(\tau,Y)$, $X$ and $\overline{y}$ be as in the statement of the lemma and let $K$ denote the upper bound on $X(\tau_j)$ for every $j\in N$.  Pick $(y,z) \in {\cal R}$ such that $y > \overline{y} \vee K\vee x_0$ and let $\tilde{X}$ be the stationary process using the $(y,z)$-ordering policy \eqref{eq-pi-defn} having the stationary distribution for its initial distribution.  First observe that $X(0-) = x_0 \leq X(0) \leq K < y \leq \tilde{X}(0)$.  Next note that $X(t) < \tilde{X}(t)$ for all $t\geq 0$ since at the order times $\{\tau_k\}$, $X(\tau_k) \leq K < y \leq \tilde{X}(\tau_k)$ and it is not possible for $X$ to move above $\tilde{X}$ through diffusion.  We also note that at the order times $\{\theta_j: j \in \N\}$ of $\tilde{X}$, $X(\theta_j-) \leq X(\theta_j) < y = \tilde{X}(\theta_j-) < \tilde{X}(\theta_j) = z$.  

Since $G \in C^1({\cal I})$ and $\sigma$ is continuous, $|G'(x) \sigma(x)|\leq L_1$ on $[y_*,\overline{y}]$ for some $L_1 < \infty$.  By (a), $G'(x) \sigma(x) < L_2$ on $(a,y_*)$ for some $L_2 < \infty$.  Therefore for each $t > 0$
\begin{eqnarray*}
\lefteqn{\E\left[\int_0^t \left(G'(X(s)) \sigma(X(s))\right)^2\left(I_{(a,y_*)}(X(s)) + I_{(y_*,\overline{y})}(X(s)) + I_{(\overline{y},b)}(X(s))\right) \d s\right]} \\
&\qquad \leq& L_1^2 t + L_2^2 t + \E\left[\int_0^t (G'(X(s))\sigma(X(s)))^2 I_{(\overline{y},b)}(X(s))\, \d s\right] \\
&\qquad \leq& L_1^2 t + L_2^2 t + \E\left[\int_0^t (G'(\tilde{X}(s))\sigma(\tilde{X}(s)))^2 I_{(\overline{y},b)}(\tilde{X}(s))\, \d s\right] \\
&\qquad =& \left(L_1^2  + L_2^2  + \int (G'(x) \sigma(x))^2 \, \pi(x)\, \d x\right) t < \infty.
\end{eqnarray*}
As a result, $G'(X(\cdot)) \sigma(X(\cdot))$ is square-integrable on $[0,t]$ for every $t < \infty$ and hence $\widehat{M}$ is a martingale.
\end{proof}

We now prove optimality of the $(y_*,z_*)$ ordering policy in the class ${\cal A}_1$.

\begin{thm}\label{thm-verification} 
Assume Conditions \ref{diff-cnd}, \ref{cost-cnds}, and \ref{eq-AG-c0} hold.  Let $F_*$ be defined by \eqref{eq-F-optimal}, and $(y_*,z_*)$ be a minimizing pair given by \propref{F-optimizers}. Then for each $(\tau, Y) \in \mathcal A_1$, 
$$F_*  \le J(\tau, Y).$$
Moreover, the impulse policy defined by \eqref{sS-tau-def} with $(y,z) = (y_*,z_*)$ is an optimal impulse policy in the class ${\cal A}_1$; that is, $J(\tau^{*}, Y^{*})= F_{*}$.
\end{thm}

\begin{proof} The fact that $J(\tau^*,Y^*) = F_*$ follows from \propref{prop-sS-cost}.

Now let $(\tau,Y)\in {\mathcal A}_1$ be arbitrary.
 We shall prove that $F_*$ is a lower bound on $J(\tau,Y)$ when $a$ is a reflective boundary point.  The proof for the other types of boundaries does not include the integral with respect to the local time process $L_a$.  Let  $\beta_{n}$  be as in \defref{A1-def}. Then applying Dynkin's formula to the function $G$ of  \eqref{eq-G-fn} yields
 \begin{displaymath}
\begin{aligned}
\E_{x_0}& [G(X(t \wedge \beta_{n}))] \\ 
    & =  G(x)+ \E_{x_0}\biggl[ \int_0^{t\wedge \beta_{n}}  AG(X(s))  \d  s
   +  \sum_{k=1}^\infty I_{\{\tau_k \le  t\wedge \beta_{n}\}}   BG(X(\tau_k-),X(\tau_k)) \\
	  & \qquad \qquad \qquad \qquad + \int_0^{t\wedge \beta_n} G'(X(s))\, dL_a(s)\biggl]  \\
  & \ge G(x) + \E_{x_0}\biggl[ \int_0^{t\wedge \beta_{n}}   [F_*- c_0(X(s))]  \d  s
   -  \sum_{k=1}^\infty I_{\{\tau_k \le  t\wedge \beta_{n}\}}   c_1 (X(\tau_k-) ,X(\tau_k))\biggl] \\
	  & \qquad \qquad \qquad \qquad + G'(a)\, \E_{x_0}[L_a(t\wedge \beta_n)] ,
 \end{aligned}
\end{displaymath} where the inequality follows from Proposition \ref{prop-G-qvi}.  Rearranging the terms and dividing both sides by $t$,  it follows that
\begin{equation}\label{eq1-verification}
\begin{aligned}
\lefteqn{\mbox{$\frac{1}{t}$} \E_{x_0}[  F_*  ( t \wedge \beta_{n}  ) - G(X(t \wedge \beta_{n})) + G'(a) L_a(t\wedge \beta_n)] }  \\  
& \quad  \le -\mbox{$\frac{1}{t}$}  G(x )  + \mbox{$\frac{1}{t}$}  \E_{x_0} \biggl[  \int_0^{t\wedge \beta_{n}}    c_0(X(s))\d s  +  \sum_{k=1}^\infty I_{\{\tau_k  \leq  t\wedge \beta_{n}\}}   c_1 (X(\tau_k-) ,X(\tau_k )) \biggl] \\
  &\quad  \le -\mbox{$\frac{1}{t}$}  G(x )  + \mbox{$\frac{1}{t}$}  \E_{x_0} \biggl[  \int_0^{t}    c_0(X(s))\d s  +  \sum_{k=1}^\infty I_{\{\tau_k  \leq  t\}}   c_1 (X(\tau_k-) ,X(\tau_k )) \biggl] .
\end{aligned}
\end{equation} 
Note that the second inequality of \eqref{eq1-verification} follows because both $c_{0}$ and $c_{1}$ are nonnegative.

Using the monotone convergence theorem, we have both $\lim_{n\to \infty} \E_{x}[t\wedge \beta_{n}] =t$ and $\lim_{n\rightarrow \infty} \E_{x_0}[L_a(t\wedge \beta_n)] = \E_{x_0}[L_a(t)]$. Thus it follows that
\begin{displaymath}\begin{aligned}
\limsup_{n\rightarrow \infty} &\mbox{$\frac{1}{t}$} \E_{x_0}[F_{*}(t\wedge \beta_{n})  - G(X(t \wedge \beta_{n})) + G'(a) L_a(t\wedge \beta_n)] \\ 
& = F_{*} - \liminf_{n\to\infty} \mbox{$\frac{1}{t}$} \E_{x_0}[G(X(t \wedge \beta_{n}))] + \mbox{$\frac{G'(a)}{t}$} \E_{x_0}[L_a(t)] \\ 
& \le -\mbox{$\frac{1}{t}$} G(x)  + \mbox{$\frac{1}{t}$} \E_{x_0} \biggl[  \int_0^{t}    c_0(X(s))\d s  +  \sum_{k=1}^\infty I_{\{\tau_k  \leq  t\}}   c_1 (X(\tau_k-) ,X(\tau_k )) \biggl].
\end{aligned}\end{displaymath}
Then taking the limits superior as $t \rightarrow \infty$ and using \eqref{eq-transversality} and \eqref{trans-local}, we obtain
\begin{displaymath}\begin{aligned}
F_{*} & \le \limsup_{t\to\infty} \biggl(-\mbox{$\frac{1}{t}$} G(x )  + \mbox{$\frac{1}{t}$} \E_{x_0} \biggl[  \int_0^{t}    c_0(X(s))\d s  +  \sum_{k=1}^\infty I_{\{\tau_k  \leq  t\}}   c_1 (X(\tau_k-) ,X(\tau_k )) \biggl]\biggr)\\
& = J(\tau, Y).
\end{aligned}\end{displaymath}
This completes the proof.
\end{proof}

\section{$(s,S)$ Optimality in ${\cal A}$ for Certain Models} \label{sect:examples}
We now extend the above results for two inventory models.  The first uses a negatively drifted Brownian motion to model the demand.  It allows back-orders with piecewise linear holding and back-order cost rates and ordering costs consisting of fixed plus proportional costs.  The second model represents demands using a geometric Brownian motion so no shortage ever occurs.  The holding cost rate function is nonlinear and the ordering cost function has fixed and level-dependent costs.  The extensions in both cases rely on cost comparisons.

\subsection{Drifted Brownian motion inventory model}\label{sect-dbm}
We begin by examining the classical model that has been studied by \cite{sulem-86}, \cite{DaiY-13-average}, and many others.  We then examine a drifted Brownian motion process with reflection at $\{0\}$.  In \sectref{sect:refl-dBm} we introduce a class of non-Markovian policies and show  some of these policies have smaller cost than an optimal $(s,S)$ ordering policy.  This new policy is not in the class ${\cal A}_1$.  
Finally, we then interpret the reflection in \sectref{sect:just-in-time} as ``just-in-time'' ordering that orders so as to fulfill unmet demand when the inventory is depleted and we provide a simple necessary and sufficient condition under which the solely just-in-time policy incurs a smaller cost than the optimal $(s,S)$ policy.  Again, this policy is not in the class ${\cal A}_1$.

\subsubsection{Classical model}
In the absence of ordering, the inventory level process $X_0$ satisfies
$$\d X_0(t) = -\mu\d t + \sigma\d W(t), \quad X_0(0) = x_0\in \I:=(-\infty, \infty)$$
in which $\mu, \sigma > 0$ and $W$ is a standard Brownian motion process.  The generator is $Af(x) = \frac{\sigma^2}{2} f''(x) - \mu f'(x)$ acting on $f \in C^2(\R)$.  We have $s(x) = e^{2\mu x /\sigma^{2}}$ and   $m(x)= \sigma^{-2} e^{-2\mu x /\sigma^{2}}$.  Therefore the scale and speed measures are respectively given by
\begin{equation} \label{dbm-measures} 
S[l,x] = \frac{\sigma^{2}}{2\mu} \left(e^{2\mu x/\sigma^{2} } - e^{2\mu l/\sigma^{2} } \right) \quad \mbox{ and } \quad M[l,x] = \frac{1}{2\mu} \left(e^{-2\mu l/\sigma^{2} } - e^{-2\mu x/\sigma^{2} } \right)
\end{equation}   for any $[l,x]\subset \I$. It is easy to see that $-\infty$ is attracting and unattainable while $\infty$ is nonattracting and unattainable.  Thus both $-\infty$ and $\infty$ are natural. This verifies Condition \ref{diff-cnd}.

To specify the cost structure, observe the model includes both holding and back-order costs as well as ordering costs.  Define
$$c_0(x) = \left\{\begin{array}{rl}
-c_b\, x, & \quad x < 0, \\
 c_h\, x, & \quad x \geq 0,
\end{array}\right.$$
in which $c_b > 0$ denotes the back-order cost rate per unit of inventory per unit of time and similarly, $c_h > 0$ is the holding cost rate.  The ordering costs are taken to be comprised of both fixed and proportional costs.  Let $k_1 > 0$ denote the fixed cost and $k_2$ denote the cost per unit ordered; thus $c_1(y,z) = k_1 + k_2(z-y)$, which is defined for all $(y,z) \in \overline{\cal R}$. 

This model was first introduced in \cite{bather-66} and revisited in \cite{sulem-86}, where it is shown that an $(s,S)$-policy minimizes the long-run average cost.  

A straightforward computation using  \eqref{eq-psi-fn} (with $C=0$)  yields $\zeta(x) = \frac{1}{\mu}\, x$ and a similar calculation  using \eqref{eq-g0-fn} (again with $C=0$)  yields
\begin{equation} \label{dbm-g0-def}
g_0(x) = \begin{cases}
-\frac{c_b}{2\mu}\, x^2 - \frac{\sigma^2 c_b}{2\mu^2}\, x + \frac{\sigma^4(c_b+c_h)}{4\mu^3} \left(\exp\set{\frac{2\mu}{\sigma^2} x} - 1\right), & \quad x < 0, \\
\frac{c_h}{2\mu}\, x^2 + \frac{\sigma^2 c_h}{2\mu^2}\, x, & \quad x \geq 0.
\end{cases}
\end{equation} Note that $g_{0} \in C^{2}(\I)$.

The following lemma can be verified in a straightforward manner.
\begin{lem} Condition \ref{cost-cnds} is satisfied.
\end{lem}

\begin{lem}\label{lem-F-optimal-dbm}
There is a unique minimizing pair $(y_*,z_*)$ of $F$ with $y_* < z_*$,  where 
\begin{equation} \label{F-def-dbm}
F(y,z) = \frac{c_1(y,z) + Bg_0(y,z)}{B\zeta(y,z)}.
\end{equation}
\end{lem}

\begin{proof}
Since both Conditions \ref{diff-cnd} and \ref{cost-cnds} are satisfied, and noting that both boundaries $-\infty$ and $\infty$ are  natural,  by Proposition \ref{F-optimizers}, there exists a pair $(y_*,z_*)$ with $-\infty < y_{*} < z_{*} < \infty$ such that $F_{* } = \inf\{F(y,z): y < z\} = F(y_{*}, z_{*})$. Moreover, the first order optimality condition   holds:
\begin{equation}\label{eq-F*yz*-dbm}
F_{*} = \mu (k_{2} + g_{0}'(y_{*})) =  \mu (k_{2} + g_{0}'(z_{*})).
\end{equation}
Hence we have $g_{0}'(y_{*}) = g_{0}'(z_{*})$.  Note that \begin{displaymath}
g_{0}'(x) =\begin{cases}
   -\frac{c_{b}}{\mu} x - \frac{\sigma^{2}c_{b}}{2\mu^{2}} + \frac{\sigma^{2}(c_{b}+c_{h})}{2\mu^{2}} e^{2\mu x /\sigma^{2}},   & \text{ if } x < 0, \\
  \frac{c_{h}}{\mu} x + \frac{\sigma^{2}c_{h}}{2\mu^{2}},     & \text{ if } x\ge 0.
\end{cases}
\end{displaymath}
Thus we can verify directly that $g_{0}'(x)$ is strictly decreasing on $(-\infty, \overline x)$ and strictly  increasing on $(\lbar x, \infty)$, where $\lbar x:= \frac{\sigma^{2}}{2\mu}\ln \frac{c_{b}}{c_{b} + c_{h}} < 0$ with  
$$g_0'(\lbar x) =  -  \mbox{$\frac{c_b \sigma^2}{2\mu^2} \ln \frac{c_b}{c_b + c_h}$} > 0. $$  
In particular, it follows that $y_{*} < \lbar x < z_{*}$ and that the pair $(y_{*} , z_{*})$ is unique.
\end{proof}

Observe that $y_* < 0$ so the $(y_*,z_*)$-ordering policy waits until a sufficient amount of inventory is on back-order before placing an order of size $z_*-y_*$.  The order-to level $z_*$ may be either positive or negative.

\begin{lem}\label{lem-G-qvi-dbm}
Let $(y_{*}, z_{*})$ be as in Lemma \ref{lem-F-optimal-dbm}.  With reference to \eqref{eq-G-fn}, define 
\begin{equation} \label{eq-G-dbm}
G(x) = \begin{cases}
 k_1 + k_2 (z_* - x)   
   + g_{0} (z_{*}) - \mbox{$\frac{F_{*}}{\mu}$} z_*, & \text{ if } x \leq y_*, \\
 -\frac{c_b}{2\mu}\, x^2 - \frac{\sigma^2 c_b}{2\mu^2}\, x + \frac{\sigma^4(c_b+c_h)}{4\mu^3} \left(\exp\set{\frac{2\mu}{\sigma^2} x} - 1\right) - \mbox{$\frac{F_{*}}{\mu}$} x,   & \text{ if  } y_* < x \leq 0, \\
 \frac{c_h}{2\mu}\, x^2 + \frac{\sigma^2 c_h}{2\mu^2}\, x - \mbox{$\frac{F_{*}}{\mu}$} x,   & \text{ if  } 0 < x. 
\end{cases} \end{equation} 
Then $G\in C^{1}(\I) \cap C^{2}(\I-\{ y_{*}\})$ and satisfies the system \eqref{eq-G-qvis}.
\end{lem}

\begin{proof} 
In view of Proposition \ref{prop-G-qvi}, it is enough to show that \cndref{eq-AG-c0} is satisfied.  Note that for all $x < y_{*} < 0$,  we have $AG(x) + c_{0}(x) = \mu k_{2} - c_{b} x,$ which   is strictly decreasing on $(-\infty, y_{*})$. This verifies \cndref{eq-AG-c0} and hence completes the proof.
\end{proof}

We consider a new class of admissible ordering policies for which the analysis is sufficient to prove optimality of the $(y_*,z_*)$ policy in the class ${\cal A}$.  This new class will be shown to be a subclass of ${\cal A}_1$ for this model.

\begin{defn} \label{A2-def}
Let ${\cal A}_2$ be the set of $(\tau,Y) \in {\cal A}$ such that the post-order inventory levels $\{X(\tau_k): k \in \N\}$ are uniformly bounded above.    
\end{defn}
Notice that the collection of $(s,S)$ policies form a subset of ${\cal A}_2$.  We need the following lemma in order to verify the conditions of \propref{add-pol}.

\begin{lem}\label{lem-X+L2}
Let $x_0$ denote the initial inventory level, $(\tau,Y) \in {\cal A}_2$ and $X$ be the resulting inventory process. Define the process $X^+$ by $X^+(t) = X(t) \vee 0$ for all $t \geq 0$.  Then $X^+ \in L^2([0,t]\times \Omega)$ for each $t \geq 0$, with 
\begin{equation} \label{X-plus-L2-bd}
\E\left[\int_0^t (X^+(s))^2\, \d s\right] \leq \widehat{K} t
\end{equation}
for some positive constant $\widehat{K}$ which depends on $\mu$, $\sigma$, $y$ and $z$ but not on $t$.
\end{lem}

\begin{proof} Let $(\tau,Y) \in {\cal A}_2$ and let $0 < K< \infty$ denote an upper bound on the post-order inventory levels.  Let $t > 0$ be fixed. 

Select $y$ and $z$ with $x_0 \vee K < y < z$.  Let $\widehat{X}$ denote the stationary inventory process that uses the $(y,z)$-ordering policy of \eqref{sS-stationary} with initial distribution given by the stationary density $\pi$ of \eqref{eq-pi-defn}-\eqref{norm-const}.  Let $\{\theta_j: j\in \N\}$ denote the order times for the $(y,z)$-policy.  Note that $\widehat{X}(t) \geq y$ for all $t \geq 0$.  For this example, 
\begin{equation} \label{pi-def-dbm}
\pi(x) =  \left\{\begin{array}{cl}
0 & \quad \mbox{for } -\infty < x \leq y, \\
\frac{1}{z-y} \left(1 - e^{-(2\mu/\sigma^2)(x-y)}\right), & \quad \mbox{for } y < x \leq z, \rule[-12pt]{0pt}{12pt} \\
\frac{e^{(2\mu z/\sigma^2)} - e^{(2\mu y/\sigma^2)}}{z-y}\cdot e^{-(2\mu/\sigma^2)x}, & \quad \mbox{for } z < x.
\end{array}  \right.
\end{equation}

Observe $\pi>0$ on $(y,\infty)$ and $\pi=0$ on $(-\infty,y]$.  Since $x_0 \vee K < y$, it follows that $X^+(0-) \leq X^+(0) < \widehat{X}(0)$.  Moreover, at each order time $\tau_k$, $X^+(\tau_k) \leq K$ so again $X^+(\tau_k) < \widehat{X}(\tau_k)$.  Finally, for each $k$, observe that on the inter-order interval $(\tau_k,\tau_{k+1})$ both processes $X$ and $\widehat{X}$ evolve according to the same drifted Brownian motion, except at times $\theta_j \in (\tau_k,\tau_{k+1})$ when $\widehat{X}$ increases.  Hence both $X(t) < \widehat{X}(t)$ for all $t\geq 0$ and, more importantly, $X^+(t) < \widehat{X}(t)$ for all $t \geq 0$.

It therefore follows that for any $ t \ge 0$,  
\begin{equation}\label{eq-X+inL2}
\begin{aligned}
 \E_{x_0} [(X^{+}(t))^{2}] &  \le  \E_{\pi} [(\widehat X(t))^{2}] =  \int_{-\infty}^{ \infty} x^{2} \pi (x) \d x \\ 
    &  =  \int_y^z  \frac{x^2 }{z-y} \Big(1 - e^{-(2\mu/\sigma^2)(x-y)}\Big)\d x  \\ & \qquad +  \int_z^\infty  \left(\mbox{$\frac{e^{(2\mu z/\sigma^2)} - e^{(2\mu y/\sigma^2)}}{z-y}$}\right) x^2 e^{-(2\mu/\sigma^2)x} \d x \\
    &   = \widehat{K} < \infty,
\end{aligned}
\end{equation}
where $\widehat{K} = \widehat{K} (\mu, \sigma, y, z) > 0 $ is independent of $t$. Integrating this bound over the interval $[0,t]$ establishes \eqref{X-plus-L2-bd}.   
\end{proof}

 \begin{prop}  \label{prop-bdd-right-locations-dbm}
For the drifted Brownian motion inventory model, ${\cal A}_2 \subset {\cal A}_1$ and hence $F_{*} \le J(\tau, Y)$ for all $(\tau, Y )\in \A_2$. 
\end{prop}
 
\begin{proof} 
Let $(\tau,Y) \in {\cal A}_2$ and let $X$ be the resulting inventory process.  The inequality holds whenever $J(\tau,Y) = \infty$ so assume $J(\tau,Y) < \infty$.  We verify the conditions of \defref{A1-def}.

The conditions of \lemref{lem:mg} are easily verified with $\overline{y}=0$ and the stationary density given by \eqref{pi-def-dbm} so $\widehat{M}$ defined as in \eqref{eq-M-process} by $\widehat{M}(t) : = \int_{0}^{t} G'(X(s)) \sigma \d W(s)$ is a martingale.  Thus in view of Remark \ref{rem-M-mg-transversality}, \propref{add-pol}, and Theorem \ref{thm-verification} the result will follow if we can show there exist positive constants $C_1$ and $C_2$ such that $\E[G(X(t))] \leq C_1\E[c_0(X(t))] + C_2$.
Now observe 
\begin{align*}
& \E_{x_{0}}[G(X(t))] \\
&= \E_{x_{0}}\left[G(X(t))\( I_{\{X(t) \le y_*\}} + I_{\{ y_* \leq X(t)  \le  0 \}} + I_{\{  X(t)  > 0\}}\)\right] \\ 
   & \le \E_{x_{0}}\! \left[ - k_2 X(t) I_{\{X(t) \le y_*\}} + (g_0(X(t))- F_* \zeta(X(t))) [I_{\{ y_* \leq X(t)  \le  0 \}} + I_{\{  X(t)  > 0\}}] \right] \\ 
   & \le  \E_{x_{0}}\biggl[ - k_2 X(t) I_{\{X(t) \le y_*\}} + C_1    
   + \left[\mbox{$\frac{ c_h}{2  \mu} X(t)^2 + \(\frac{c_h  \sigma^2  }{2 \mu^2} - 
   \frac{ F_*  }{\mu} \) X(t)\biggl] I_{\{  X(t)  > 0\}}$}  \right] \\
   & \le   C_1  + K_2 \E_{x_{0}}\left[  |X(t)|\right] \\
	 & = C_1 + C_2 \E[c_0(X(t))],
\end{align*} where the last inequality follows from \eqref{eq-X+inL2}.  Notice $C_1$ and $C_{2}$ are positive constants independent of $t$, establishing the result.
\end{proof}

The following proposition is from \cite{HeYZ-15}.
\begin{prop} \label{prop-comparison-dbm}
For any $(\tau,Y)\in \mathcal A$ and any $\e > 0$,  there exists an ordering policy $(\theta,Z)\in {\cal A}_2$ such $J(\theta,Z) \leq J(\tau,Y) + \e$.
\end{prop}

It follows from this proposition that for each $(\tau,Y) \in {\cal A}$, $J(\tau,Y)$ is in the closure of $\{J(\tau,Y): (\tau,Y) \in {\cal A}_2\}$ from which the final theorem immediately follows.

\begin{thm}
Let $F$ be defined by \eqref{F-def-dbm} and let $(y_*,z_*)$ be as in \lemref{lem-F-optimal-dbm}.  Then the $(y_*,z_*)$ ordering policy is optimal in the class ${\cal A}$.
\end{thm}

\subsubsection{Drifted Brownian motion with reflection at $\{0\}$} \label{sect:refl-dBm}
We now consider the model in which the inventory level process is a drifted Brownian motion with reflection at $\{0\}$.  The holding cost rate function is $c_0(x) = k_3 x$ for $x \in {\cal E} = [0,\infty)$ and the ordering cost function is $c_1(y,z) = k_1 + k_2(z-y)$ for $(y,z) \in {\cal R}$.  We note that no back-orders are allowed and $c_0$ does not satisfy \eqref{c0-lim-at-a} at the boundary $\{0\}$.  New to the model is a cost $k_5$ that is charged per unit of reflection.  The generator remains $Af(x) = \frac{\sigma^2}{2} f''(x) - \mu f'(x)$, the jump operator is $Bf(y,z) = f(z) - f(y)$, the scale and speed measures continue to be given by \eqref{dbm-measures}, and $g_0(x) = \frac{c_h}{2\mu}\, x^2 + \frac{\sigma^2c_h}{2\mu^2}\, x$ and $\zeta(x) = \frac{1}{\mu}\, x$ for $x \in {\cal E} = [0,\infty)$.  

Let $(\tau,Y)\in {\cal A}$, $X$ denote the resulting inventory process and $L_0$ denote the local time process of $X$ at $\{0\}$.  The long-term average cost associated with this policy is 
\begin{equation} \label{dBm-lta-refl}
J(\tau,Y) = \begin{array}[t]{l} 
\displaystyle \limsup_{t\rightarrow \infty} \mbox{$\frac{1}{t}$} \E_{x_0}\left[\int_0^t c_0(X(s))\, \d s \right. \\
\qquad \quad \qquad \displaystyle \left. + \sum_{k=1}^\infty I_{\{\tau_k \leq t\}} c_1(X(\tau_k-),X(\tau_k)) + k_5 L_0(t)\right]
\end{array}
\end{equation}
and by the extended It\^{o} formula, for $f \in C^2({\cal E})$,
\setlength{\arraycolsep}{0.5mm}
\begin{eqnarray*}
f(X(t)) = f(x_0) &+& \int_0^t Af(X(s))\, \d s + \sum_{k=1}^\infty I_{\{\tau_k\leq t\}} Bf(X(\tau_k-),X(\tau_k)) \\
&+& \int_0^t Cf(X(s))\, \d L_0(s) + \int_0^t \sigma f'(X(s))\, \d W(s)
\end{eqnarray*}
in which $C$ is the reflection operator $Cf(x) = f'(x)$.

Briefly considering $(y,z)$ ordering policies, note that the order occurs before any reflection when $y=0$ and for $y > 0$, the process $X$ never reaches $\{0\}$.  Thus no reflection occurs under any $(y,z)$ policy.  As a result, the long-term average cost of this class of policies is again given by $F$ of \eqref{F-def-dbm}.  It is straightforward to verify that the optimizing pair of $F$ is $(y_*,z_*) = \left(0,\sqrt{\frac{2k_1\mu}{c_h}}\right)$ and the optimal cost in this class of policies is $F_* = \sqrt{2k_1 \mu c_h} + k_2\mu + \frac{\sigma^2 c_h}{2\mu}$.

We now introduce and analyze a non-Markovian ordering policy.  Consider the ``delayed $(s,S)$ ordering policy with trigger $\ss$'' in which each order is placed to raise the process $X$ to the level $S$ but the subsequent order requires $X$ to first hit a level $\ss<s$ and then is placed when $X$ next hits $s$.  Hitting $\ss$ acts as a trigger in the same way as a knock-in boundary in option pricing.  We may think of $X$ being in ``Phase 1'' prior to the triggering event and in ``Phase 2'' following.  

More specifically, we choose $(y,z) \in {\cal R}$ so $y > 0$ and set the trigger level to be $\ss = 0$.  This policy induces a stationary distribution on ${\cal I}$.  One way by which to find the stationary density is to augment a ``Phase process'' $\Theta$ taking values in $\{1,2\}$ and observe that the pair process $(X,\Theta)$ is Markovian.  The generator of the process acts on functions $f \in C^2({\cal E})$ and $g \in {\cal M}\{1,2\}$ and is $\widehat{A}[fg](x,\theta) = Af(x)g(\theta)$.  The jump operator is 
$$\widehat{B}[fg](x,\theta) = f(x) [g(3-\theta) - g(\theta)] I_{(\ss,1)}(x,\theta) + [f(z)g(3-\theta) - f(x) g(\theta)]I_{(y,2)}(x,\theta)$$
and the reflection operator is $\widehat{C}[fg](x,\theta) = f'(x) g(\theta)I_{\{2\}}(\theta)$ with this being active only when $X(t)=0$.   For the pair process $(X,\Theta)$, with reference to the conditions \eqref{stat-cnd} and \eqref{stat-cnd-refl}, the stationarity condition to be satisfied is
\begin{eqnarray*}
0 &=& \int_{{\cal I}\times \{1,2\}} \widehat{A}[fg](x,\theta)\, \mu_0(\d x\times \d \theta) + \int_{{\cal R}\times \{0,1\}} \widehat{B}[fg](x,\theta)\, \mu_1(\d x\times \d\theta) \\
& & +\; \int_{{\cal E} \times \{1,2\}} \widehat{C}[fg](x,\theta)\, \mu_2(\d x\times \d\theta) 
\end{eqnarray*}
which must hold for all $f \in C^2({\cal E})$ and $g \in {\cal M}(\{1,2\})$.  Since the $\Theta$ process is used merely to make the pair process Markovian, the stationary density for the $X$ process can be computed from this stationarity condition, resulting in
\begin{eqnarray*}
\pi(x) &=& \alpha_1 \left[\mbox{$\frac{1}{\mu}$}\left(e^{2\mu(y-x)/\sigma^2} - 1\right)I_{[0,y]}(x) 
+ \mbox{$\frac{1}{\mu}$} \left(1-e^{-2\mu x/\sigma^2}\right)I_{[0,z]}(x) \right. \\
& & \qquad \left. +\; \mbox{$\frac{1}{\mu}$}\left(e^{2\mu z/\sigma^2} - 1\right) e^{-2\mu x/\sigma^2} I_{(z,\infty)}(x)\right]
\end{eqnarray*}
in which $\alpha_1$ is the normalizing constant $\left(\frac{(z-y)}{\mu} + \frac{\sigma^2}{2\mu^2}(e^{2\mu y/\sigma^2}-1)\right)^{-1}$.  In addition, the long-term average frequency of orders from $y$ to $z$ is $\alpha_1$ and hence the expected cycle length is $\frac{(z-y)}{\mu} + \frac{\sigma^2}{2\mu^2}(e^{2\mu y/\sigma^2}-1)$.  Moreover the expected long-term average amount of reflection using the local time of $X$ is $\kappa=\alpha_1 \frac{\sigma^2}{2\mu}\left(e^{2\mu y/\sigma^2} - 1\right)$; the fact that $\kappa > 0$ means that \eqref{trans-local} of \defref{A1-def} fails so the delayed $(y,z)$ policy with trigger $\ss=0$ is not in ${\cal A}_1$.  Using a renewal argument as in \propref{prop-sS-cost} establishes that the expected long-term average cost associated with this policy is
\begin{eqnarray*}
\widetilde{F}(y,z) := J(\tau,Y) &=& \frac{c_1(y,z) + Bg_0(y,z) + \left(\frac{\sigma^4 c_h}{4\mu^3} + \frac{\sigma^2 k_5}{2\mu}\right)\left(e^{2\mu y/\sigma^2} -1\right)}{B\zeta(y,z) + \frac{\sigma^2}{2\mu^2}\left(e^{2\mu y/\sigma^2} - 1\right)} \\
&=&  \frac{c_1(y,z) + Bg_0(y,z)}{B\zeta(y,z)} \left(\frac{1 + \frac{\left(\frac{\sigma^4 c_h}{4\mu^3} + \frac{\sigma^2 k_5}{2\mu}\right)\left(e^{2\mu y/\sigma^2} -1\right)}{c_1(y,z) + Bg_0(y,z)}}{1 + \frac{\frac{\sigma^2}{2\mu^2}\left(e^{2\mu y/\sigma^2} - 1\right)}{B\zeta(y,z)}}\right)\;.
\end{eqnarray*}
Observe that when $y=0$, the first expression reduces to the long-term average cost associated with a $(0,z)$ ordering policy as it should since ordering occurs before any reflection.  

Consider first policies in which $y > 0$.  The second expression for $\widetilde{F}$ indicates that the cost of the delayed $(y,z)$ ordering policy with trigger $\ss=0$ is smaller than the cost for the standard $(y,z)$ policy when the second factor is less than $1$.  This relation holds when
\begin{displaymath}
 \frac{\left(\frac{\sigma^4 c_h}{4\mu^3} + \frac{\sigma^2 k_5}{2\mu}\right)\left(e^{2\mu y/\sigma^2} -1\right)}{c_1(y,z) + Bg_0(y,z)} <  \frac{\frac{\sigma^2}{2\mu^2}\left(e^{2\mu y/\sigma^2} - 1\right)}{B\zeta(y,z)} 
\end{displaymath}
or, equivalently, since the cost of the $(y,z)$ policy is $F(y,z) = \frac{c_1(y,z) + Bg_0(y,z)}{B\zeta(y,z)}$, when
\begin{displaymath}
\mbox{$\frac{\sigma^{2} c_{h}}{2\mu}$} + k_{5} \mu < F(y,z).
\end{displaymath} 
Since $F(y,z) \ge F_{*} = \sqrt{2 k_{1} c_{h} \mu} + k_{2} \mu + \frac{\sigma^{2} c_{h}}{2 \mu}$, a sufficient condition for the cost of the delayed $(y,z)$ policy with trigger $0$ to be smaller than the $(y,z)$ policy for each $0 < y < z$ is then given by 
$$k_{5} < k_{2} + \sqrt{\mbox{$\frac{2 k_{1} c_{h}}{\mu}$}}.$$ 

When $y=0$, a different analysis is required since $\widetilde{F}(0,z) = F(0,z)$.  Computing the partial derivative of $\widetilde{F}$ with respect to $y$, the denominator is positive and the numerator reduces to 
\setlength{\arraycolsep}{0.5mm}
\begin{eqnarray*}
-\; \mbox{$\frac{c_{h}}{\mu^{2}}$}\, y (z-y) &-& \mbox{$\frac{c_{h}\sigma^{2} y + c_{h} \mu   (z^{2} - y^{2}) + 2 \mu^{2} k_{1}}{2\mu^{3}}$} (e^{2\mu y/\sigma^2} -1) \\
&-& (k_{2} - k_{5})   \left[\mbox{$\frac{z-y }{\mu}$} e^{\frac{2\mu}{\sigma^2} y} +  \mbox{$\frac{\sigma^{2}}{2\mu^{2}}$} (e^{2\mu y/\sigma^2} -1) \right].
\end{eqnarray*}
The first two terms are negative and the last term is also negative when $k_2 > k_5$.  Thus when $k_2 > k_5$, $\frac{\partial \widetilde{F}}{\partial y}(0,z) < 0$ and increasing the ordering level $y$ from $0$ for the delayed policy decreases the long-term average cost.  In particular, taking $z=z_*$, this analysis shows that for some $y > 0$ a delayed $(y,z_*)$ ordering policy with trigger $\ss=0$ has lower cost than the optimal $(0,z_*)$ policy.

In summary, $k_5 < k_2$ is a sufficient condition for the non-Markovian delayed $(y,z)$ ordering policy with trigger $\ss=0$ to have lower cost than the standard $(s,S)$ policy with $(s,S)=(y,z)$.

\subsubsection{Just-in-time ordering vs. $(s,S)$ ordering} \label{sect:just-in-time}
The inventory model is again given by a drifted Brownian motion process with drift rate $-\mu < 0$ and with reflection at $\{0\}$.  We interpret the reflection at $\{0\}$ of the process to be an ordering policy which places an order so as to only meet the demand when the stocks are depleted; such orders can be considered ``emergency'' orders for a ``just-in-time'' inventory policy.  The cost structure continues to be as in \sectref{sect:refl-dBm}.  In particular, the ordering cost function is $c_1(y,z) = k_1 + k_2(z-y)$ and the cost per unit of reflection is $k_5$.  The cost rate $k_5$ then represents the cost per unit for an emergency delivery such as is provided by overnight shipping.  It is natural to have $k_5 > k_2$ since an ``emergency'' order should be more costly than a ``regular'' order.  Write $k_5 = k_2 + \tilde{k}_5$ with $\tilde{k}_5 > 0$; thus $\tilde{k}_5$ represents the cost premium (per unit) for an emergency order.  

As in \sectref{sect:refl-dBm}, the optimal $(y_*,z_*)$ ordering policy has optimal levels $y_*=0$ and $z_*= \sqrt{\frac{2k_1\mu}{c_h}}$, and optimal cost $F_* = \sqrt{2k_1 c_h \mu} + k_2 \mu + \frac{\sigma^2 c_h}{2\mu}$.  Note the optimal policy orders amount $z_*$ every time the inventory level falls to $0$ so no reflection occurs; that is, there are no just-in-time ordering costs. 

Consider the {\em solely}\/ just-in-time policy which only places orders so as to fulfill unmet demand.  The resulting inventory level process is the drifted Brownian motion with reflection at $\{0\}$.  It is straightforward to verify using \eqref{stat-cnd-refl} (see also p.~129 of \cite{boro:02}) that the stationary density for this process is
$$\pi(x) = \mbox{$\frac{2\mu}{\sigma^2}$}\, e^{-2\mu x/\sigma^2} \qquad \mbox{for } x \geq 0$$
and the expected long-term average amount of just-in-time ordering (local time at $\{0\}$) is $\kappa = \mu$.  Again, $\kappa > 0$ indicates that the just-in-time policy is not in the class ${\cal A}_1$.  It now follows that the long-term average cost is 
$\widehat{F} = \mbox{$\frac{\sigma^2 c_h}{2\mu}$} + k_5 \mu$.

Comparing $F_*$ and $\widehat{F}$, we see that the just-in-time ordering policy is better than the optimal $(y_*,z_*)$ policy exactly when
$$(k_2 + \tilde{k}_5)\mu = k_5 \mu < k_2\mu + \sqrt{2k_1 c_h \mu} = \left(k_2 + \mbox{$\frac{c_h}{\mu} \sqrt{\frac{2 k_1 \mu}{c_h}}$} \right) \mu = \left(k_2 + \mbox{$\frac{c_h z_*}{\mu}$}\right)\mu;$$
that is, when the cost premium $\tilde{k}_5$ per unit for an emergency order is smaller than $\frac{c_h z_*}{\mu}$.  Since $\frac{z_*}{\mu}$ is the expected cycle length of the $(0,z_*)$ policy, the just-in-time policy is preferred to the optimal $(0,z_*)$ policy exactly when the expected holding cost per unit over a cycle exceeds the cost premium per unit for the emergency orders.

 \subsection{Geometric Brownian motion storage model} \label{sect:gBM}
Without any ordering, the inventory process is a geometric Brownian motion satisfying the stochastic differential equation
$$dX_0(t) = - \mu X_0(t)\d t + \sigma X_0(t) \d W(t), \qquad X(0) = x_0\in \I=(0,\infty),$$
in which $\mu, \sigma > 0$ and $W$ is a standard Brownian motion process; the drift rate is negative.  The generator of this process is $Af(x) = \frac{\sigma^2}{2}\/ x^2 f''(x) - \mu x f'(x)$ which acts on functions $f \in C^2(\R ^+)$.  The jump operator remains $Bf(y,z) = f(z) - f(y)$ for a jump from $y$ to $z>y$.

For this model $s(x) = x^{{2\mu}/{\sigma^{2}}}$ and $m(x) = \sigma^{-2}x^{-2-{2\mu}/{\sigma^{2}}}$ and therefore the scale and speed measures are respectively given, for $[l,x] \subset \I$, by 
\begin{displaymath}\begin{aligned}
S[l,x] & = \mbox{$\frac{\sigma^{2}}{2\mu+ \sigma^{2}}$} \left[x^{ 1+  {2\mu}/{\sigma^{2}}}-l^{1+  2\mu/\sigma^{2}}\right] \quad \mbox{and} \\
M[l,x] & =  \mbox{$\frac{1}{2\mu + \sigma^{2}}$} \left[ l^{-1-2\mu/\sigma^{2}} - x^{-1-2\mu/\sigma^{2}}\right].
\end{aligned}\end{displaymath} 
In particular, it follows that $0$ is an attracting natural boundary while $\infty$ is a non-attracting natural boundary. This verifies Condition \ref{diff-cnd}.

In the following subsection, the cost formulation is given for the standard model under consideration.  Later subsections then examine variations on this formulation.

\subsubsection{Standard cost model}

The holding cost rate function is $c_0(x) = k_3 x + k_4 x^\beta$ in which $k_3, k_4 > 0$ and $\beta < 0$.  Observe that the term $k_3 x$ incurs large costs for large inventory levels whereas $k_4 x^\beta$ is costly as the level approaches $0$.  We examine models in Sections~\ref{case-1}-\ref{case-3} in which some of the coefficients $k_2$, $k_3$ and $k_4$ equal $0$.  

Using \eqref{eq-g0-fn} and \eqref{eq-psi-fn} (with $C=1$) respectively, we determine 
\begin{displaymath} 
\begin{aligned}
& g_0(x) = \mbox{$\frac{k_3}{\mu} (x -1) - \frac{k_4}{\tilde\rho(\beta)} (x^\beta -1)$}, \quad  & x \in \I,\\ 
& \zeta(x) = \mbox{$\frac{2}{2\mu+\sigma^2}$} \ln(x), \quad  &  x \in \I. 
\end{aligned} 
\end{displaymath}
Observe that $\tilde\rho(\beta) := \frac{\sigma^2}{2}\/ \beta^2 - (\mu+ \frac{\sigma^2}{2}) \beta > 0$.

Turning to the ordering costs, fix $\eta$ with $0 < \eta \leq 1$ and define $c_1$ on ${\cal R}$ by
\begin{equation} \label{gbm2-c1-def}
c_1(y,z) = k_1 + k_2(z^\eta - y^\eta).
\end{equation}  
When $\eta < 1$, the function $x \mapsto x^\eta$ is strictly concave.  Note that 
$$k_2 (z^\eta - y^\eta) = k_2\cdot \frac{(z^\eta - y^\eta)}{(z - y)} \cdot (z-y)$$
so the cost per unit ordered $k_2\,\frac{(z^\eta-y^\eta)}{(z-y)}$ is adjusted based on the quantity $y$ on hand prior to ordering and the quantity $z$ following ordering.  In particular, when $1 < y < z$, $z^\eta - y^\eta < z - y$ so by adopting this model, the manufacturer encourages his customers to maintain large inventory levels.  The parameter $\eta$ can be viewed as a measure of the amount of importance the manufacturer places on his customers having large inventories, with $\eta$ close to $0$ providing more of a discount for large inventory levels.

\begin{lem} \label{lem-cnd2-gbm2}
Condition \ref{cost-cnds} is satisfied.
\end{lem}
\begin{proof}
It is obvious that the function $c_{0}$ is inf-compact. Moreover, \eqref{c0-M-integrable} and \eqref{infinite-dbl-intgrl-at-b} can be verified using direct calculations. The conditions \eqref{decr-cost} and  \eqref{eq-c1-equal-displacement} for the function $c_{1} $ are trivially satisfied.
\end{proof}

With this ordering cost function $c_1$, the function $F$ to be minimized is
\begin{equation} \label{F-def}
F(y,z) = \frac{k_1 + k_2(z^\eta-y^\eta) + \frac{k_3}{\mu} (z-y) - \frac{k_4}{\tilde\rho(\beta)} (z^\beta - y^\beta)}{\frac{2}{2\mu+\sigma^2}(\ln(z)-\ln(y))}.
\end{equation}

\begin{lem} \label{lem-F-gbm2}
There is a unique minimizing pair $(y_*,z_*)$ of $F$ with $y_* < z_*$. 
\end{lem}

\begin{proof}
Since Conditions~\ref{diff-cnd} and \ref{cost-cnds} are satisfied, \propref{F-optimizers} establishes the existence of a minimizing pair $(y_*,z_*) \in {\cal R}$.  We therefore seek to show that the optimizing pair is unique.  The first-order optimality conditions lead immediately to the system
$$F_* = \mbox{$(\mu+\frac{\sigma^2}{2})[\frac{k_3}{\mu} y + k_2 \eta y^\eta + \frac{k_4(-\beta)}{\tilde{\rho}(\beta)} y^\beta]$} = \mbox{$(\mu+\frac{\sigma^2}{2})[\frac{k_3}{\mu} z + k_2 \eta z^\eta + \frac{k_4(-\beta)}{\tilde{\rho}(\beta)} z^\beta]$}.$$
Define the function $h$ on $\I$ by 
\begin{equation} \label{gbm2-h-def}
h(x) = \mbox{$\frac{k_3}{\mu} x + k_2 \eta x^\eta + \frac{k_4(-\beta)}{\tilde{\rho}(\beta)} x^\beta$}.
\end{equation}  
Then  
\begin{eqnarray} \label{gbm2-h-prime} \nonumber 
h'(x) &=& \mbox{$\frac{k_3}{\mu} + k_2 \eta^2 x^{\eta-1} - \frac{k_4 \beta^2}{\tilde{\rho}(\beta)} x^{\beta-1}$} \\
&=& \left[\mbox{$\frac{k_3}{\mu} x^{1-\beta} + k_2 \eta^2 x^{\eta-\beta} - \frac{k_4 \beta^2}{\tilde{\rho}(\beta)}$}\right] x^{\beta-1} 
\end{eqnarray}
so $h'$ is negative on the interval $(0,\widehat{x})$ and positive on $(\widehat{x},\infty)$, where $\widehat{x}$ is the root of $h'$ which is also the root of the coefficient of $x^{\beta-1}$ in the last expression of \eqref{gbm2-h-prime}.  Thus each level set of $h$ above its minimal value consists of two points, establishing that the minimizing pair $(y_*,z_*)$ is unique.
\end{proof}

\begin{lem}\label{lem-G-qvi-gbm2} 
Let $(y_{*},z_{*}) \in \R$ be as in Lemma \ref{lem-F-gbm2}.  With reference to \eqref{eq-G-fn}, define 
\setlength{\arraycolsep}{0.5mm}
$$G(x):= \begin{cases} 
\begin{array}{rcl} 
-k_2 x^\eta + k_{1} + k_{2} z_*^\eta +  \mbox{$\frac{k_3}{\mu} (z_* -1)$} &-& \mbox{$\frac{k_4}{\tilde\rho(\beta)} (z_*^\beta -1)$} \\
&-& F_{*} \mbox{$\frac{2}{2\mu+\sigma^2}$} \ln(z_*)
\end{array}   & \text{ if } 0< x \le y_{*}, \\
 \mbox{$\frac{k_3}{\mu} (x -1) - \frac{k_4}{\tilde\rho(\beta)} (x^\beta -1)$} - F_{*}\mbox{$\frac{2}{2\mu+\sigma^2}$} \ln(x)   & \text{ if  } x > y_{*}. 
\end{cases}$$ 
Then $G$ satisfies the hypotheses of \propref{prop-G-qvi}. 
\end{lem}

\begin{rem}
As a result of \lemref{lem-G-qvi-gbm2}, $G \in C^{1}(\I)\cap C^{2}(\I - \{ y_{*}\})$ and satisfies the system \eqref{eq-G-qvis}.  Hence in light of \remref{rem-qvi}, $G$ and $F_*$ satisfy the QVI \eqref{qvi} for the geometric Brownian motion inventory model.
\end{rem}

\begin{proof}
By the comments at the beginning of this section, \cndref{diff-cnd} is satisfied and \lemref{lem-cnd2-gbm2} establishes that \cndref{cost-cnds} is satisfied.  It therefore remains to show that the function $x \mapsto AG(x) + c_0(x)$ is decreasing on the interval $(0,y_*)$.  For simplicity, define the function $\tilde{h} = AG + c_0$.

On the set $(0,y_*)$, the function $G$ is $G(x) = -k_2 x^\eta + k_1 + k_2 z_*^\eta + g_0(z_*) - F_* \zeta(z_*)$.  Thus on the set $(0,y_*)$, 
\begin{eqnarray*}
\tilde{h}(x) &=& \mbox{$\frac{\sigma^2}{2} x^2 G''(x) - \mu x G'(x) + c_0(x)$} \\
&=& k_3 x + k_2 \mbox{$\left(\mu \eta + \frac{\sigma^2 \eta(1-\eta)}{2}\right) x^\eta + k_4 x^\beta$}
\end{eqnarray*}
and hence
\begin{eqnarray} \label{gbm2-tilde-h-prime} \nonumber
\tilde{h}'(x) &=& k_3 + k_2 \mbox{$\left(\mu \eta^2 + \frac{\sigma^2 \eta^2(1-\eta)}{2}\right) x^{\eta-1} - k_4 (-\beta) x^{\beta-1}$} \\
&=& \left[\mbox{$k_3 x^{1-\beta} + k_2\left(\mu \eta^2 + \frac{\sigma^2 \eta^2(1-\eta)}{2}\right) x^{\eta-\beta} - k_4(-\beta)$}\right] x^{\beta-1}.
\end{eqnarray}
Observe that $\tilde{h}'$ strictly increases from $\tilde{h}'(0+)=-\infty$ to $\tilde{h}'(\infty)=k_3 > 0$.  Let $\tilde{x}$ denote the  root of $\tilde{h}'$, which is also the root of the coefficient of $x^{\beta-1}$ in the last expression of \eqref{gbm2-tilde-h-prime}. Then $\tilde{h} = AG + c_0$ is strictly decreasing on the interval $(0,\tilde{x})$ and is strictly increasing on $(\tilde{x},\infty)$.

By the definitions of $\widehat{x}$ and $\tilde{x}$ as the roots of $h'$ and $\tilde{h}'$, respectively, or more precisely as the roots of the coefficients of $x^{\beta-1}$ in \eqref{gbm2-h-prime} and \eqref{gbm2-tilde-h-prime}, respectively, we have
$$\mbox{$\frac{k_3}{\mu} \widehat{x}^{1-\beta} + k_2 \eta^2 x^{\eta-\beta}  - \frac{k_4 \beta^2}{\tilde{\rho}(\beta)}$} = 0 = \mbox{$k_3 \tilde{x}^{1-\beta} + k_2\left(\mu \eta^2 + \frac{\sigma^2 \eta^2(1-\eta)}{2}\right) \tilde{x}^{\eta-\beta} - k_4(-\beta)$}.$$
Multiplying the left expression (for $0$) by $\mu$ and rearranging the terms yields
\begin{eqnarray*}
\mbox{$k_3\left(\tilde{x}^{1-\beta} - \widehat{x}^{1-\beta}\right) + k_2\mu \eta^2\left(\tilde{x}^{\eta-\beta} - \widehat{x}^{\/\eta-\beta}\right) + \frac{k_2 \sigma^2 \eta^2(1-\eta)}{2} \tilde{x}^{\eta-\beta}$} &=& \mbox{$k_4 (-\beta) - \frac{k_4 \beta^2\mu}{\tilde{\rho}(\beta)}$} \\
&=& \mbox{$\frac{k_4 \sigma^2 \beta^2(1-\beta)}{2\tilde{\rho}(\beta)}$}
\end{eqnarray*}
or, equivalently, 
\begin{equation} \label{eq5} \begin{array}{r}
k_3\left(\tilde{x}^{1-\beta} - \widehat{x}^{1-\beta}\right) + \left(k_2\mu \eta^2 + \frac{k_2 \sigma^2 \eta^2(1-\eta)}{2}\right)\left(\tilde{x}^{\eta-\beta} - \widehat{x}^{\/\eta-\beta}\right) \\
= \frac{k_4 \sigma^2 \beta^2(1-\beta)}{2\tilde{\rho}(\beta)} - \frac{k_2 \sigma^2 \eta^2(1-\eta)}{2} \widehat{x}^{\eta-\beta}.
\end{array}
\end{equation}
Observe that since $\widehat{x}$ is a root of the coefficient of $x^{\beta-1}$ in \eqref{gbm2-h-prime}, 
$$\mbox{$- k_2 \eta^2 \widehat{x}^{\eta-\beta} = \frac{k_3}{\mu} \widehat{x}^{1-\beta} - \frac{k_4 \beta^2}{\tilde{\rho}(\beta)}$}$$
so making this substitution in \eqref{eq5} yields
$$\begin{array}{rcl}
k_3\left(\tilde{x}^{1-\beta} - \widehat{x}^{1-\beta}\right) &+& (k_2\mu \eta^2 + \frac{k_2 \sigma^2 \eta^2(1-\eta)}{2})\left(\tilde{x}^{\eta-\beta} - \widehat{x}^{\/\eta-\beta}\right) \rule[-10pt]{0pt}{10pt} \\
&=& \frac{k_4 \sigma^2 \beta^2(1-\beta)}{2\tilde{\rho}(\beta)} - \frac{k_4 \sigma^2 \beta^2(1-\eta)}{2\tilde{\rho}(\beta)} + \frac{k_3 \sigma^2 (1-\eta)}{2\mu}\widehat{x}^{\eta-\beta} \rule[-10pt]{0pt}{10pt} \\
&=& \frac{k_4 \sigma^2 \beta^2(\eta-\beta)}{2\tilde{\rho}(\beta)} + \frac{k_3 \sigma^2 (1-\eta)}{2\mu}\widehat{x}^{\eta-\beta} > 0.
\end{array}$$
Hence $\tilde{x} > \widehat{x}$.  Since $\widehat{x}$ is the minimizer of $h$ in \eqref{gbm2-h-def}, it follows that $y_* < \widehat{x} < \tilde{x}$ and therefore $h=AG+c_0$ is decreasing on $(0,y_*)$.
\end{proof}

For the geometric Brownian motion inventory model, it is again sufficient to restrict attention to the class ${\cal A}_2$ in \defref{A2-def}.

\begin{prop}  \label{prop-bdd-right-locations-2}
Let $(\tau,Y) \in {\cal A}_2$ with $J(\tau,Y) < \infty$.  For each $n\in \N$, define the stopping time $\beta_n = \inf\{t\geq 0: X(t) \geq n\}$.  Then
\begin{equation}
\label{eq-trans-G-gbm}
 \liminf_{t\rightarrow \infty} \lim_{n\rightarrow \infty} \mbox{$\frac{1}{t}$} \E_{x_0}\left[  G(X(t\wedge\beta_n))\right] = 0
\end{equation}
and hence ${\cal A}_2 \subset {\cal A}_1$.  Consequently we have $F_{*} \le J(\tau, Y)$.  Moreover, with $(y_*,z_*)$ being the minimizing pair of \lemref{lem-F-gbm2}, the policy $(\tau^*,Y^*)$ defined in \eqref{sS-tau-def} is optimal in the class ${\cal A}_2$.
\end{prop}

\begin{proof}
Let $(\tau,Y) \in {\cal A}_2$, let $\beta_n$ be as in the statement of the proposition and assume that $K_1$ is a uniform bound on the post-order inventories.  Define $\widehat{M}(t) : = \int_{0}^{t} G'(X(s))\sigma X(s) \d W(s)$. Observe that 
\begin{eqnarray*}
\lefteqn{\E\left[\int_0^{t\wedge \beta_n} (\sigma X(s) G'(X(s))^2 I_{\{X(s) \in (0,y_*)\}}\, \d s\right]} \\ 
&=& \E\left[\int_0^{t\wedge \beta_n} k_2^2 \sigma^2 \eta^2 (X(s))^{2\eta} I_{\{X(s) \in (0,y_*)\}}\, \d s\right] \leq k_2^2 \sigma^2 \eta^2 t \cdot (y_*)^{2\eta}. 
\end{eqnarray*}
The localizations satisfy $X(s)I_{\{X(s) > y_*\}} \leq n$ for all $s \leq \beta_n$ and hence $\{\widehat{M}(t\wedge\beta_n): t \geq 0\}$ is a martingale for each $n$.

We now show that \eqref{eq-trans-G-gbm} holds.  Recall from \eqref{eq-G-bdd-above} that $G(x)\ge -\widehat{\kappa}$ for all $x\in \I$ so 
\begin{equation} \label{eq-G-mean-2parts}
\begin{array}{rcl}
 \E_{x_0}\!\left[G(X(t\wedge\beta_n)) + \widehat{\kappa}\right] &=& \E_{x_0}\!\left[(G(X(t)) + \widehat{\kappa}) I_{\{t\leq \beta_n\}}\right] \rule[-10pt]{0pt}{10pt} \\
& & +\; \E_{x_0}\!\left[ (G(X(n)) + \widehat{\kappa}) I_{\{t >\beta_n\}}\right].
\end{array}
\end{equation}
Observe that the first summand on the right-hand side of \eqref{eq-G-mean-2parts} is monotone increasing in $n$ so it follows from the fact that $\beta_n \rightarrow \infty$   a.s.  that
\begin{equation}
\label{eq1-G-gbm}
\lim_{n\rightarrow \infty}\E_{x_0}\left[ (G(X(t)) + \widehat{\kappa}) I_{\{t\leq \beta_n\}}\right] = \E_{x_0}\left[G(X(t)) + \widehat{\kappa}\right].
\end{equation}

We now determine the asymptotic behavior of the second summand on the right-hand side of \eqref{eq-G-mean-2parts}.
Applying It\^{o}'s formula along with the optional sampling theorem using the  scale function $S(x) : = \int_{1}^{x} s(y) \d y = \frac{\sigma^{2}}{2\mu + \sigma^{2}} x^{1+ 2\mu/\sigma^{2}}$, we obtain for each $t > 0$
$$\begin{aligned}
S(X(t\wedge \beta_n)) =  S(x_0)&  + \int_0^{t\wedge \beta_n}  AS(X(s)) \d s
+ \sum_{k=1}^\infty   I_{\{\tau_k \leq t\wedge\beta_n\}} BS(X(\tau_k-),X(\tau_k)) \\ &
+ \widehat{M}(t\wedge\beta_{n}).
\end{aligned}$$
The regular integral vanishes because $AS(x) = 0$, and since $\beta_n$ is a localizing time, taking expectations eliminates the stochastic integral.  Hence
\begin{equation} \label{eq-S-mean-two-parts}
\E_{x_0}\left[S (X(t\wedge \beta_n))\right] = S(x_{0}) +  \E_{x_0}\left[\sum_{k=1}^\infty   I_{\{\tau_k \leq t\wedge\beta_n\}} BS(X(\tau_k-),X(\tau_k))\right].
\end{equation}
Since $X(\tau_k) \leq K_1$ for every $k$, and observe $X(\tau_{k}-) > 0$, it follows that 
$$BS(X(\tau_k-),X(\tau_k)) \leq  \mbox{$\frac{\sigma^{2}}{2\mu + \sigma^{2}}$} K_{1}^{1+ 2\mu/\sigma^{2}}=: K_{2} < \infty.$$   
This estimate provides the bound
\begin{eqnarray} \label{eq-BS-estimate-gbm}\nonumber
\E_{x_0}\Biggl[\sum_{k=1}^\infty  I_{\{\tau_k \leq t\wedge\beta_n\}} BS(X(\tau_k-),X(\tau_k))\Biggr]
&\leq&  K_2\E_{x_0}\Biggl[\sum_{k=1}^\infty   I_{\{\tau_k \le  t\}} \Biggr] \\
&=&  K_{2} \sum_{k=1}^\infty  \E_{x_{0}}[I_{\{\tau_k \le  t\}}].
\end{eqnarray}
Since $(\tau,Y)$ induces a finite cost, noting the facts that $c_{0} \ge 0$ and $c_{1} \ge k_{1} > 0$, we have 
\begin{eqnarray*}
\limsup_{t\to\infty} \mbox{$\frac{1}{t}$} \E_{x_{0}} \Biggl[\sum_{k=1}^{\infty} I_{\{\tau_{k} \le t\}} k_{1}  \Biggr] &\leq& \limsup_{t\to\infty} \mbox{$\frac{1}{t}$} \E_{x_{0}} \Biggl[\sum_{k=1}^{\infty} I_{\{\tau_{k} \le t\}} c_{1}(X(\tau_{k}-),X(\tau_{k})) \Biggr] \\
&\leq& J(\tau, Y) < \infty.
\end{eqnarray*}
Therefore there exists some $T > 0$ such that the right-hand side of \eqref{eq-BS-estimate-gbm} is bounded  by $K_3 t$ for all $t \geq T$, in which $K_{3}$ is independent of $t$ and $n$.  On the other hand, since $S$ is positive, we have $S(n) \E_{x_{0}}[I_{\{t>  \beta_{n}  \}}] \le \E_{x_0}[S (X(t\wedge \beta_n))].$  Putting these estimates into \eqref{eq-S-mean-two-parts}, we obtain
$$ \P_{x_{0}}\set{t > \beta_{n}} \le \frac{S(x_{0}) + K_{3} t}{S(n)}.$$

  Observe that there exist positive constants $K_4$ and $K_5$ such that 
\begin{equation} \label{G-bd-gbm}
0 \le G(x) + \widehat{\kappa} \leq K_4 + K_5 x
\end{equation} 
for all $x \in(0,\infty)$.  As a result,
\begin{displaymath}
\E_{x_0}\big[ (G(X(\beta_{n})) + \widehat{\kappa}) I_{\{t >\beta_n\}}\big] \le (K_{4} + K_{5}n) \P_{x_{0}}\set{t > \beta_{n}} \le  (K_{4} + K_{5}n) \frac{S(x_{0}) + K_{3} t}{S(n)} \to 0,
\end{displaymath}
 as $n \rightarrow \infty$.  This, together with \eqref{eq-G-mean-2parts} and \eqref{eq1-G-gbm}, implies that 
$$\lim_{n\rightarrow \infty} \E_{x_0}\left[G(X(t\wedge\beta_n)) + \widehat{\kappa}\right] = \E_{x_0}\left[G(X(t)) + \widehat{\kappa}
\right]$$  
and hence 
\begin{equation}
\label{eq-limit-G-n}
\lim_{n\to\infty}\E_{x_{0}}[G(X(t\wedge\beta_n))]= \E_{x_{0}}[G(X(t))].
\end{equation}

As a result of \eqref{eq-G-bdd-above}, we have $\liminf_{t\to\infty}\frac{1}{t}\E_{x_{0}}[G(X(t))] \ge 0$. Further observe that \eqref{G-bd-gbm} implies the existence of positive constants $C_1$ and $C_2$ such that $G(x) \leq C_1 c_0(x) + C_2$ for every $x \in (0,\infty)$ and hence $\E[G(X(t))] \leq C_1\E[c_0(X(t))] + C_2$ for all $t\geq 0$.  Following the argument in the proof of \propref{add-pol}, the transversality condition \eqref{eq-trans-G-gbm} is established and hence ${\cal A}_2 \subset {\cal A}_1$.

The hypotheses of \thmref{thm-verification} hold and therefore both the inequality $F_{*}\le J(\tau, Y)$ and the optimality of the $(\tau^*,Y^*)$ in the restricted class ${\cal A}_2$ of ordering policies is proven. 
 \end{proof}

As in the drifted Brownian motion model, a comparison result extends the optimality of the $(y_*,z_*)$-ordering policy to all ordering policies in ${\cal A}$.  In contrast with \propref{prop-comparison-dbm}, the structure of the geometric Brownian motion model allows a stronger comparison result, namely, that the cost of the modified policy is less than the cost of the original policy on a path-by-path basis.  The proof of this comparison is similar to that of Proposition 4.7 of \cite{HelmesSZ-14} which considers a discounted criterion but the differences are significant enough to warrant a complete exposition.  This proof is given in the appendix.

\begin{prop} \label{prop-comparison-gbm2}
Let $(\tau,Y)\in \mathcal A$ and let $X$ denote the resulting inventory process.  Then there exists an admissible ordering policy $(\theta,Z) \in {\cal A}_2$ with resulting inventory process $\tilde{X}$, such that for each $t \geq 0$,
\begin{equation} \label{cost-comparison-gbm2}
\begin{array}{l} \displaystyle
\int_0^t c_0(\tilde{X}(s))\, \d s + \sum_{j=1}^\infty c_1(\tilde{X}(\theta_j-),\tilde{X}(\theta_j)) I_{\{\theta_j \leq t\}} \\
\displaystyle \qquad \qquad \leq \int_0^t c_0(X(s))\, \d s + \sum_{k=1}^\infty c_1(X(\tau_k-),X(\tau_k)) I_{\{\tau_k \leq t\}}.
\end{array}
\end{equation}
Consequently, $J(\theta,Z) \leq J(\tau,Y)$.
\end{prop}

It is now necessary to return to the original inventory control problem over the set ${\mathcal A}$ of admissible ordering policies.

\begin{thm}\label{thm-gbm}
Let $(y_*,z_*)$ be as in \lemref{lem-F-gbm2} and let $(\tau^*,Y^*)$ be defined by \eqref{sS-tau-def}.  Then $(\tau^*,Y^*)$ is optimal in the class ${\cal A}$.
\end{thm}
\begin{proof}
This theorem follows from Propositions \ref{prop-bdd-right-locations-2} and \ref{prop-comparison-gbm2} directly.
\end{proof}
\medskip

To conclude we show that the optimality of an $(s,S)$ policy for the geometric Brownian motion inventory model critically depends on the parameters in the holding cost and ordering cost functions $c_{0}(x)= k_{3} x + k_{4} x^{\beta}$ and $c_{1}(y,z) = k_{1} + k_{2}(z^\eta-y^\eta)$, respectively.  In particular, we demonstrate that \cndref{cost-cnds}(a) is sufficient but not necessary for the existence of an optimal $(s,S)$ ordering policy.  In the ensuing subsections, we fix the constant $k_1>0$ and consider particular choices of $k_2$, $k_3$ and $k_4$. 

\subsubsection{Parameter variation: $k_4 =0$, $k_2, k_3 > 0$.} \label{case-1}
When $k_{4} =0$, small inventories are not penalized.  We can verify directly that 
\begin{displaymath}
\E_{x_{0}} \biggl[ \frac{1}{T} \int_{0}^{T} c_{0}(X_{0}(t)) \d t\biggr]  = \frac{1}{T} \int_{0}^{T}  k_{3} x_{0} e^{-\mu t} \d t = \frac{k_{3} x_{0}}{ \mu T} (1- e^{-\mu T}) \to 0 ,
\end{displaymath}
as $T \to \infty$, where $X_{0}$ is the uncontrolled inventory process with initial condition $x_{0} > 0$. This shows that when $k_{4} =0$, the ``no-order'' policy is optimal.  This result is not surprising since $0$ is an attracting boundary point so almost all paths converge to $0$ and each order simply delays this convergence at a higher cost in both ordering and holding costs.

For each $L > 0$, the set $\{x\in \I: c_0(x) \le L\} = (0, \frac{L}{k_3}]$ is not compact. Thus the inf-compactness assumption for the function $c_0$  in Condition \ref{cost-cnds} is violated when $k_4 =0$.  On the other hand,  Proposition \ref{sS-stationary} implies that the $(s, S)$-policy   with $s= y$ and $S= z$ has long-term expected order frequency $\frac{1}{\zeta(z) - \zeta(y)}$. Therefore, the proof of  Proposition \ref{prop-sS-cost} shows that the cost of the $(s,S)$-policy with $s = y$ and $S=z$ defined in \eqref{sS-tau-def} is
$$ F(y,z) = \frac{k_1 + k_2(z^\eta-y^\eta) + \frac{k_3}{\mu}(z-y)}{\frac{2}{2\mu + \sigma^{2}} (\ln z - \ln y)}.$$
Note that  for any $z$ fixed, we have $ \lim_{y \to  0} F(y,z)  = 0$.  But for any fixed pair $(y,z) \in {\mathcal R}$ with $0 < y <z $, we have $F(y,z) > 0$. Thus the function $F$ has no minimizing pair $(y_{*},z_{*}) \in \mathcal{R}$.

\subsubsection{Parameter variation: $k_{2} = k_{3}=0$, $k_4>0$. } \label{case-2}
When  $k_{2} = k_{3}=0$, as argued before, Condition \ref{diff-cnd} holds but not Condition \ref{cost-cnds}.  In addition,  the cost of the $(s,S)$-policy with $s = y$ and $S=z$  is given by
$$F(y,z) = \frac{k_1   -  \frac{k_{4}}{\tilde \rho(\beta)}(z^{\beta} - y^{\beta})}{\frac{2}{2\mu + \sigma^{2}} (\ln z - \ln y)}.$$
Since $F$ is positive and can be made arbitrarily small by taking $z$ sufficiently large with $y$ fixed.  But for any fixed pair $(y,z) \in {\mathcal R}$ with $0 < y <z $, we have $F(y,z) > 0$. Thus the function $F$ has no minimizing pair $(y_{*},z_{*}) \in \mathcal{R}$. This problem has no optimal inventory control policy.

\begin{rem}
These first two variations demonstrate that the function $F$ does not attain its infimal value of $0$ and no $(s,S)$ policy is optimal.  The next variation indicates that \cndref{diff-cnd} and \cndref{cost-cnds} are not necessary in that \cndref{cost-cnds}(a) fails but we are still able to prove existence of an optimal $(s,S)$ ordering policy.
\end{rem}

\subsubsection{Parameter variation: $k_{3} =0$, $k_2, k_4 > 0$.} \label{case-3}
Next, we consider the case when $k_{3}=0$ so only small inventories are penalized. Again, we immediately see that the function $c_{0}(x) = k_{4} x^{\beta}$ is not inf-compact and hence Condition \ref{cost-cnds} (a) does not hold. Moreover, straightforward calculations reveal that for each $y> 0$,
\begin{align}
\label{eq-gbm-c0-integral}& \int_{y}^{\infty} c_{0}(v) \d M(v) = \mbox{$\frac{k_{4}}{2\mu + \sigma^{2}-\beta \sigma^{2}}$}\, y^{\beta - 1 - 2\mu/\sigma^{2}} < \infty,       \\
    &   \int_{y}^{\infty} \int_{u}^{\infty} c_{0}(v) \d M(v) \d S(u) = \mbox{$\frac{k_{4}}{(2\mu + \sigma^{2}-\beta \sigma^{2}) (-\beta)} $}\, y^{\beta} < \infty.
\end{align}
Thus \eqref{c0-M-integrable} holds while \eqref{infinite-dbl-intgrl-at-b} is violated.  In addition, by Proposition \ref{prop-2.4-h-integral} and the calculations in the proof of \propref{cor-g0-psi-fns},  we have the following representation of the holding costs:
  \begin{equation}\label{eq-c0-integral-gbm-bad-exm}\begin{aligned}
\E_{z} & \biggl[ \int_{0}^{\tau_{y}} c_{0}(X_{0}(s)) \d s\biggr] \\
& =  2 \int_{y}^{z} [S(v) - S(y)] m(v) c_{0}(v) \d v + 2 [S(z) - S(y)] \int_{z}^{\infty} m(v) c_{0}(v) \d v \\
&= - \mbox{$\frac{k_{4}}{\tilde \rho(\beta)}$} (z^{\beta} - y^{\beta}).
\end{aligned}\end{equation}

\begin{prop}\label{prop-gbm-bad-exm}
Define the function $F$ on ${\cal R}$ by
\begin{equation}
\label{def-F-gbm-bad-exm}
F(y,z) : = \frac{k_{1}+ k_{2} (z^\eta-y^\eta) - \frac{k_{4}}{\tilde \rho(\beta )} (z^{\beta} - y^{\beta})}{ \frac{2}{2\mu + \sigma^{2}} (\ln z - \ln y)}.
\end{equation} 
Then there exists a unique pair $(y_{*}, z_{*}) \in \mathcal {R}$ such that $F$ attains its minimum value at $(y_{*}, z_{*})$.  Moreover, the $(s,S)$-policy defined in \eqref{sS-tau-def} with $s= y_{*}$ and $S= z_{*}$ is an optimal inventory control policy.
\end{prop}

\begin{proof}To show the existence of an optimizing pair $(y_{*}, z_{*})\in {\cal R}$, we follow the approach of the proof of \propref{F-optimizers} by examining the asymptotic values of $F$ as $(y,z)$ goes to the boundaries of ${\cal R}$.
\begin{enumerate}
\item[\em (i)] {\em The diagonal boundary $z=y$.}\/  The limit of $F(y,z)$ as $(z-y) \rightarrow 0$ is infinite due to the fixed cost exactly as in the analysis of (i) of the proof of \propref{F-optimizers}. 
\item[\em (ii)] {\em The upper boundary $z=\infty$ with $y$ fixed.}\/  For each $y\in \I$ fixed, $\lim_{z\to \infty} F(y,z) = \infty$ since $\lim_{z\to \infty} \frac{z^\eta}{\ln z}  =\infty$.
\item[\em (iii)] {\em The boundary point $(\infty,\infty)$.}\/  Notice that $ - \frac{k_{4}}{\tilde \rho(\beta )} (z^{\beta} - y^{\beta}) > 0$.  Using the changes of variables, $z=e^u$ and $y=e^v$ and applying the mean value theorem
  \begin{displaymath}
F(y,z) > \frac{k_{2} (z^\eta-y^\eta) }{\frac{2}{2\mu + \sigma^{2}} (\ln z - \ln y)} = \frac{k_{2} (e^{\eta u}-e^{\eta v}) }{\frac{2}{2\mu + \sigma^{2}} (u - v)}  =  k_{2}\eta(\mu + \mbox{$\frac{\sigma^{2}}{2}$}) e^{\eta \xi},
\end{displaymath} 
where $v < \xi <u$. In particular, when $(y,z) \to (\infty, \infty)$, we have $(u,v) \to (\infty,\infty)$ and hence $\xi \to \infty$ as well.  Therefore the assertion  $\lim_{(y,z) \to (\infty, \infty)} F(y,z) = \infty$ follows.
\item[\em (iv)] {\em The left boundary $y=0$ with $z$ fixed.}\/  For each $z\in \I$ fixed, both $\lim_{y\to 0} y^{\beta} = \infty$ and $-\lim_{y\to 0}  \ln y = \infty$ so we can apply L'H\^opital's rule to obtain
  \begin{displaymath}
\lim_{y\to 0} F(y,z) = \lim_{y\to 0}  \frac{ \frac{k_{4}}{\tilde \rho(\beta )} \beta y ^{\beta-1}}{ \frac{2}{2\mu + \sigma^{2}} (-\frac{1}{y})  } = \lim_{y\to 0} \mbox{$\frac{k_{4} \beta (2\mu + \sigma^{2})}{2\tilde \rho(\beta )}$}  y^{\beta} = \infty.
\end{displaymath}
\item[\em (v)] {\em The boundary point $(0,\infty)$.}\/  In this case, $\lim_{(y,z)\rightarrow (0,\infty)} F(y,z) = \infty$ follows from cases (ii) and (iii) above using exactly the same argument as in the corresponding case (v) of the proof of \propref{F-optimizers}.
\item[\em (vi)] {\em The boundary point $(0,0)$.}\/ Similar to case (iii), make the changes of variable $z:= e^{-u}$ and $y:= e^{-v}$ with $0 < u  < v$. Then using the mean value theorem, we have
\begin{displaymath}\begin{aligned}
F(y,z) & = \frac{k_{1 } + k_{2} (e^{-u} - e^{-v}) - \frac{k_{4}}{\tilde \rho(\beta)} (e^{-u \beta} - e^{-v \beta})}{\frac{2}{2\mu + \sigma^{2}}  ((-u) - (-v))} \\
&  >  \mbox{$\frac{k_{4}(2\mu + \sigma^{2})}{2 \tilde \rho(\beta)}$}  \cdot\frac{e^{-v\beta} -e^{-u \beta}}{v-u}  \\
&=  \mbox{$\frac{k_{4}(2\mu + \sigma^{2})}{2 \tilde \rho(\beta)}$}  (-\beta) e^{-\xi \beta},
\end{aligned}\end{displaymath} 
where $u < \xi < v$. When $(y,z) \to (0, 0)$, $(u,v) \to (\infty,\infty)$ and hence $\xi\to \infty$. Therefore it follows that $\lim_{(y,z) \to (0, 0)} F(y,z) =  \infty$ as desired.
\end{enumerate}

 In view of Cases (i)--(vi) above, we conclude that there exists a pair  $(y_{*}, z_{*}) \in \mathcal {R}$ such that the function $F(y,z)$ defined in \eqref{def-F-gbm-bad-exm} attains its minimum value at $(y_{*}, z_{*})$:
$$F_{*}: = \inf\{F(y,z): (y,z) \in \mathcal{R} \} = F(y_{*}, z_{*}) > 0.$$

 Moreover, using   similar arguments as those  for the proofs of Lemmas \ref{lem-F-gbm2} and  \ref{lem-G-qvi-gbm2}, we can show that the minimizing pair $(y_{*}, z_{*}) \in \mathcal {R}$ is unique and that the function 
\begin{equation}\label{G-defn-gbm-bad-exm}
G(x) : = \begin{cases}
 k_{1} + k_{2}(z_*^\eta - x^\eta) - \frac{k_{4}}{\tilde \rho(\beta)} (z_{*}^{\beta} -1) - F_{*} \frac{2 \mu  + \sigma^{2}}{2}\ln z_{*} & \text{ if }0 < x \le y_{*}, \\
- \frac{k_{4}}{\tilde \rho(\beta)} (x^{\beta} -1) - F_{*} \frac{2 \mu  + \sigma^{2}}{2}\ln x  & \text{ if  } x > y_{*}, \rule{0pt}{15pt} \\
\end{cases}
\end{equation} 
solves the system \eqref{eq-G-qvis}.

Next we use Theorem \ref{thm-verification} to verify that the   $(s,S)$-policy with $s= y_{*}$ and $S= z_{*}$ is an optimal inventory control policy. To this end, we first notice that for any $(\tau, Y ) \in \A$ with $J(\tau,Y) < \infty$ and the corresponding controlled inventory process $X$, the process $\widehat{M}(t): = \int_{0}^{t} G' (X(s)) \sigma X(s) \d W(s)$ is a martingale:
  \begin{align*}
\E & \biggl[ \int_{0}^{t} (G' (X(s)) \sigma X(s))^{2} \d s \biggr]  \\ 
 & = \E \biggl[ \int_{0}^{t} (G' (X(s)) \sigma X(s))^{2} \(I_{\{ X(s) \le y_{*}\}} +  I_{\{ X(s) > y_{*}\}}\)\d s \biggr] \\
 & \le \E\biggl[ \int_{0}^{t} \( k_{2}^{2} \sigma^{2} \eta^2 y_{*}^{2\eta}I_{\{ X(s) \le y_{*}\}} + \Bigl[ F_{*} \mbox{$\frac{2 \mu  + \sigma^{2}}{2}$}  \sigma + \mbox{$\frac{k_{4}}{\tilde \rho(\beta)}$} \sigma X(s)^{\beta}\Bigr]^{2}  I_{\{ X(s) > y_{*}\}} \) \d s \biggr] \\
 & \le  \int_{0}^{t} \( k_{2}^{2} \sigma^{2} \eta^2 y_{*}^{2\eta} + \Bigl[ F_{*} \mbox{$\frac{2 \mu  + \sigma^{2}}{2}$}  \sigma + \mbox{$\frac{k_{4}}{\tilde \rho(\beta)}$} \sigma y_{*}^{\beta}\Bigr]^{2} \) \d s < \infty.
\end{align*} 
Consequently, in view of Remark \ref{rem-M-mg-transversality}, for any $(\tau, Y ) \in \A$, we have $F_{*} \le J(\tau, Y)$ provided that we can verify the transversality condition \eqref{eq-transversality2}.  Due to the definitions of the functions $c_{0}(x) = k_{4} x^{\beta}$ and $G$  in \eqref{G-defn-gbm-bad-exm}, we can verify directly that there exists some positive constants $C_{1}$ and $C_{2}$ so that $C_{1} c_{0} (x) + C_{2}  \ge G(x) $ for all $x\in \I$.  Proof of the transversality condition then follows as in the proof of \propref{add-pol}.

Finally we notice that Proposition \ref{sS-stationary} implies that the $(s, S)$-policy   with $s= y_{*}$ and $S= z_{*}$ has long-term expected order frequency $\frac{1}{\zeta(z_{*}) - \zeta(y_{*})}$. Therefore, by \eqref{eq-c0-integral-gbm-bad-exm} and the proof of  Proposition \ref{prop-sS-cost}, the $(s, S)$-policy has cost 
$$\frac{k_{1} + k_{2}(z_{*} - y_{*}) - \frac{k_{4}}{\tilde \rho (\beta)}(z_{*}^{\beta} - y_{*}^{\beta})}{\frac{2\mu+ \sigma^{2}}{2} (\ln z_{*} - \ln y_{*})} = F(y_{*},z_{*})= F_{*}.$$  
This completes the proof.
\end{proof}

\appendix

\section{Proof of \propref{prop-comparison-gbm2}}
The proof of this comparison result is broken into two parts.  The first proposition examines the total cost up to time $t$ of a single order of the $(\tau,Y)$ policy.  The remainder of the proof of \propref{prop-comparison-gbm2} combines the individual adjustments as in the proof of Proposition 4.7 of \cite{HelmesSZ-14} with only minor adjustments needed to account for the different ordering cost function $c_1(y,z) = k_1 + k_2(z^\eta - y^\eta)$.  We therefore refer the reader to the proof in \cite{HelmesSZ-14} and omit the details.

\begin{prop} \label{single-order-comparison}
Let $(\tau,Y) \in {\cal A}$, let $(\tau_k,Y_k)$ denote the $k^{\mbox{\footnotesize th}}$ order and let $X_k$ denote the inventory process resulting from this single order.  Then there exists a policy $(\theta_k,Z_k) = \{(\theta_{k,j},Z_{k,j}): j = 1, \ldots, m_k\} \in {\cal A}_2$ and resulting process $\tilde{X}_k$ such that \eqref{cost-comparison-gbm2} holds pathwise for each $t \geq 0$.
\end{prop}

\begin{proof}
To begin, notice that the function $c_0(x) = k_{3} x + k_{4 } x^{\beta}$ attains its minimum value at $y := \left(\frac{(-\beta)k_4}{k_3}\right)^{\frac{1}{1-\beta}}$ and that $c_{0}$ is strictly increasing on $[y, \infty)$.  Fix any $z > y$.  Define the functions 
\begin{equation} \label{l-r-fns}
\begin{array}{rcl}
\ell(m) &=& z^\eta + \left[\mbox{$\frac{k_1}{k_2}$} + (z^\eta - y^\eta)\right]m, \\
\widehat{\ell}(m) &=& \left(\frac{k_1 z^\eta}{k_2(z^\eta - y^\eta)} + z^\eta\right) m, \mbox{ and } \\
r(m) &=& y^\eta \left(\frac{z^\eta}{y^\eta}\right)^{m-1}
\end{array} 
\end{equation} 
and form the intervals ${\cal I}_m = \left[\ell(m),r(m)\right]$ and $\widehat{\cal I}_m = \left[\widehat{\ell}(m),r(m)\right]$.  Since both $\ell$ and $\widehat{\ell}$ grow linearly while $r$ grows geometrically, there exist $\overline{m}$ and $\widehat{m}$ such that
$$\bigcup_{m=\overline{m}}^\infty {\cal I}_m = \left[\ell(\overline{m}),\infty\right) \quad \mbox{ and } \quad \bigcup_{m=\widehat{m}}^\infty \left[\widehat{\ell}(\widehat{m}),\infty\right).$$
Define $L = \ell(\overline{m}) \vee \widehat{\ell}(\widehat{m}) \vee z^\eta.$  The reason for considering the particular functions $\ell$, $\widehat{\ell}$ and $r$ will become clear as we analyze the policy $(\theta_k,Z_k)$, which we now define.

The new policy $(\theta_k,Z_k)$ adjusts the jump of $X_k$ at time $\tau_k$ as follows:
\begin{itemize}
\item[(a)] if the post-order location $X_k(\tau_k) \leq L^{\frac{1}{\eta}}$, use the same order so there is no difference in the two trajectories;
\item[(b)] if $X_k(\tau_k) > L^{\frac{1}{\eta}}$ and $X_k(\tau_k-) > y$,
\begin{itemize}
\item do not order at time $\tau_k$ so the process $\tilde{X}_k(\tau_k) = X(\tau_k-)$; 
\item wait to order until the new process $\tilde{X}_k$ falls to level $y$ and order up to level $z\wedge X_k$;  
\item continue in this manner until the $\tilde{X}_k$ and $X_k$ processes coalesce;
\end{itemize}
\item[(c)] if $X_k(\tau_k) > L^{\frac{1}{\eta}}$ and $X_k(\tau_k-) \leq y$, 
\begin{itemize}
\item immediately order up to level $z$ so $\tilde{X}_k(\tau_k) = z$;
\item when $\tilde{X}_k$ hits $y$, order up to level $z\wedge X_k$; and
\item continue in this manner until the $\tilde{X}_k$ and $X_k$ processes coalesce.
\end{itemize}
\end{itemize}

Since exactly the same order is made in case (a), the costs will be the same.  We therefore need to analyze cases (b) and (c).  

The key to understanding the construction of the policy $(\theta_k,Z_k)$ and process $\tilde{X}_k$ is in understanding how the orders affect the evolution of the inventory processes.  A straightforward analysis establishes that 
$$X_k(t) = \left\{\begin{array}{lr}
x_0 \exp(-(\mu+(\sigma^2/2)) t + \sigma W(t)), & \quad 0 \leq t < \tau_k, \\
x_0 \exp(-(\mu+(\sigma^2/2)) t + \sigma W(t)) \cdot \frac{X_k(\tau_k)}{X_k(\tau_k-)}, & \quad t \geq \tau_k
\end{array}\right.$$
and the pathwise cost up to time $t$ associated with this single order is
\begin{equation} \label{X-k-cost}
\int_0^t (k_3 X_k(s) + k_4 X_k\beta(s))\, \d s + [k_1 + k_2 (X_k^\eta(\tau_k) - X_k^\eta (\tau_k-))] I_{\{\tau_k \leq t\}}.
\end{equation}

Turning to the $\tilde{X}_k$ process, we begin with case (c).  Assume that $X_k(\tau_k-) \leq y$ and $X_k(\tau_k) > L^{\frac{1}{\eta}}$.  In this situation, an order is placed at time $\theta_{k,0}:=\tau_k$ so $\tilde{X}_k(\tau_k) = z$ and $\tilde{X}_k$'s evolution follows the same geometric Brownian motion path from this position.  The second order occurs at $\theta_{k,1}= \inf\{t> \tau_k: \tilde{X}_k(t-) = y\}$.  Over the time interval $[\theta_{k,0},\theta_{k,1})$, the process $\tilde{X}_k$ is
$$\tilde{X}(t) = z \exp(-(\mu+(\sigma^2/2)) t + \sigma W(t)), \quad \theta_{k,0} \leq t < \theta_{k,1}. $$
Notice that $\tau_k = \theta_{k,0} < \theta_{k,1}$ so on the interval $[\tau_k,\theta_{k,1})$, the processes satisfy
$$\frac{\tilde{X}_k(t)}{X_k(t)} = \frac{z}{X(\tau_k)};$$
that is, the processes $\tilde{X}_k$ and $X_k$ are in constant proportion, with constant of proportionality $\frac{z}{X_k(\tau_k)} < 1$ (since $X_k(\tau_k) > L \geq z$).  Also, over the interval $[\tau_k,\theta_{k,1})$, the process moves from $\tilde{X}_k(\tau_k) = z$ to $y$; that is, $\tilde{X}_k$ is reduced by the factor $\frac{y}{z}$.  Since $\tilde{X}_k$ and $X_k$ are in constant proportion, it follows that $X_k$ is also reduced by this factor.  At time $\theta_{k,1}$, an order of size $Z_{k,1} = (z \wedge X_k(\theta_{k,1}))-y$ is placed.  When $X_k(\theta_{k,1}) \leq z$, $\tilde{X}_k$ coalesces with $X_k$; otherwise, $\tilde{X}_k(\theta_{k,1}) = z$ and again the processes $\tilde{X}_k$ and $X_k$ evolve according to the same geometric Brownian motion path from their respective positions and order $(\theta_{k,2},Z_{k,2})$ will be placed when both processes are again reduced by the factor $\frac{y}{z}$.  

Now define 
\begin{equation} \label{m-k-def-c}
m_k = \min\{j: X_k(\tau_k) \left(\mbox{$\frac{y}{z}$}\right)^{j-1} \leq z\}
\end{equation} 
and note that $m_k$ is a finite random variable which will give the number of orders needed after the initial order at time $\tau_k$ for the processes to coalesce.  Recall the definitions of $r$ in \eqref{l-r-fns} and $\overline{m}$. Observe from the definition of $m_k$ and the fact that $X_k(\tau_k) > L^{\frac{1}{\eta}} \geq (r(\overline{m}))^{\frac{1}{\eta}}$ that
\begin{equation} \label{X-tau-ub}
y \cdot \mbox{$\left(\frac{z}{y}\right)^{\overline{m}}$} < X_k(\tau_k) \leq z \cdot \mbox{$\left(\frac{z}{y}\right)^{m_k-1}$} = y \cdot \mbox{$\left(\frac{z}{y}.\right)^{m_k}$}.
\end{equation}
For $j = 1,\ldots, m_k$, define $\theta_{k,j} = \inf\{t>\theta_{k,j-1}: \tilde{X}_k(t-) = y\}$ and, for $j=1,\ldots,m_k-1$, set $Z_{k,j} = z-y$ while $Z_{k,m_k} = X_k(\theta_{k,m_k})-y$.  Then by construction, $\tilde{X}_k(t) \leq X_k(t)$ for all $t \geq 0$.  Moreover, the pathwise cost up to time $t$ of this ordering policy is
\setlength{\arraycolsep}{0.5mm}
$$ \begin{array}{l} \displaystyle
\int_0^t (k_3 \tilde{X}_k(s) + k_4 \tilde{X}_k^\beta(s))\, \d s + [k_1 + k_2(z^\eta - X_k^\eta(\tau_k-))]I_{\{\theta_{k,0} \leq t\}} \\
+ \displaystyle \sum_{j=1}^{m_k-1} [k_1 + k_2(z^\eta-y^\eta)]I_{\{\theta_{k,j} \leq t\}}  
+ [k_1 + k_2(X_k^\eta(\theta_{k,m_k}) - y^\eta)] I_{\{\theta_{k.m_k} \leq t\}}. 
\end{array}$$
Comparing this cost with \eqref{X-k-cost}, we note that since $c_0(x) = k_3 x + k_4 x^\beta$ is strictly increasing on $(y,\infty)$ and since $y \leq \tilde{X}_k(s) \leq X_k(s)$ for all $s$, it immediately follows that 
\begin{equation} \label{holding-cost-comparison-gbm2}
\int_0^t (k_3 \tilde{X}_k(s) + k_4 \tilde{X}_k^\beta(s))\, \d s \leq \int_0^t (k_3 X_k(s) + k_4 X_k^\beta(s))\, \d s.
\end{equation}
We now seek to have the $(\theta_k,Z_k)$ ordering costs be smaller than those for $(\tau_k,Y_k)$.  For this to be true, we require
\begin{eqnarray} \label{order-cost-req} \nonumber
[k_1 + k_2 (X_k^\eta(\tau_k) - X_k^\eta (\tau_k-))] I_{\{\tau_k \leq t\}} &\geq& [k_1 + k_2(z^\eta - X_k^\eta(\tau_k-))]I_{\{\theta_{k,0} \leq t\}} \\ 
& & + \sum_{j=1}^{m_k-1} [k_1 + k_2(z^\eta-y^\eta)]I_{\{\theta_{k,j} \leq t\}} \\ \nonumber 
& & + [k_1 + k_2(X_k^\eta(\theta_{k,m_k}) - y^\eta)] I_{\{\theta_{k,m_k} \leq t\}}.
\end{eqnarray}
First, since $X_k(\theta_{k,m_k}) \leq z$, $k_1 + k_2(z^\eta - y^\eta)$ is greater than the cost at the last ordering time $\theta_{k,m_k}$.  Also since $\tau_1 = \theta_{k,0} < \theta_{k,1} < \cdots < \theta_{k,m_k}$, the condition \eqref{order-cost-req} is satisfied whenever
$$k_1 + k_2 (X_k^\eta(\tau_k) - X_k^\eta (\tau_k-)) \geq [k_1 + k_2(z^\eta - X_k^\eta(\tau_k-))] + [k_1 + k_2(z^\eta-y^\eta)]m_k$$
or equivalently, whenever
\begin{equation} \label{X-tau-lb}
X_k^\eta(\tau_k) \geq z^\eta + \left[\mbox{$\frac{k_1}{k_2}$} + (z^\eta - y^\eta)\right]m_k = \ell(m_k).
\end{equation}
The for those paths for which $(\tau_k,Y_k)$ is in case (c), the ordering policy $(\theta_k,Z_k)$ has bounded post-order inventory sizes and costs less up to time $t$, for each $t \geq 0$, than the order $(\tau_k,Y_k)$.

Turning to an examination of case (b) in which $X_k(\tau_k-) > y$, there is no order at time $\theta_{k,0}=\tau_k$.  We need to consider separately the situations in which $y < X_k(\tau_k) \leq \frac{z}{y}\, X_k(\tau_k-)$ and $X_k(\tau_k) > \frac{z}{y}\, X_k(\tau_k-)$.  In the first of these situations, the $(\theta_k,Z_k)$ ordering policy will consist of a single order at time $\theta_{k,1}$ so that $\tilde{X}_k(\theta_{k,1}) = X_k(\theta_{k,1}) = X_k(\tau_k) \cdot \frac{y}{X_k(\tau_k-)}$.  The corresponding ordering cost is 
\begin{eqnarray*}
c_1(\tilde{X}_k(\theta_{k,1}-),\tilde{X}_k(\theta_{k,1}) &=& [k_1 + k_2(\tilde{X}_k^\eta(\theta_{k,1}) - y^\eta)] I_{\{\theta_{k,1} \leq t\}} \\
&=& [k_1 + k_2((X_k(\tau_k) \cdot \mbox{$\frac{y}{X_k(\tau_k-)}$})^\eta \ - y^\eta)] I_{\{\theta_{k,1} \leq t\}} \\
&=& [k_1 + k_2(X_k^\eta(\tau_k) - X_k(\tau_k-)^\eta)\cdot \mbox{$(\frac{y}{X_k(\tau_k-)}$})^\eta] I_{\{\theta_{k,1} \leq t\}} \\
&\leq& [k_1 + k_2(X_k^\eta(\tau_k) - X_k(\tau_k-)^\eta)] I_{\{\tau_k \leq t\}} \\
&=& c_1(X_k(\tau_,-),X_k(\tau_k)).
\end{eqnarray*}

Finally, when $X_k(\tau_k) > \frac{z}{y}\, X_k(\tau_k-)$, more than a single order is required in order for the $\tilde{X}_k$ process to coalesce with the $X_k$ process.  Again letting $m_k$ denote the number of orders, we must have
$$z \geq X_k(\theta_{k,m_k}) = X_k(\tau_k) \cdot \mbox{$\frac{y}{X_k(\tau_k-)} \left(\frac{y}{z}\right)^{m_k-1}$}$$
and this condition is satisfied whenever \eqref{X-tau-ub} holds.  Comparing costs, we require
\begin{equation} \label{cost-comparison-gbm2-b}
\begin{array}{rcl} \displaystyle
\sum_{j=1}^{m_k-1} [k_1 + k_2(z^\eta - y^\eta)]I_{\{\theta_{k,j}\leq t\}} &+& [k_1 + k_2(X_k^\eta(\theta_{k,m_k}) - y^\eta)]I_{\{\theta_{k,m_k}\leq t\}} \\
&\leq& [k_1 + k_2(X_k^\eta(\tau_k) - X_k^\eta(\tau_k-))]I_{\{\tau_k\leq t\}}
\end{array}
\end{equation}
Again using the fact that $X_k(\theta_{k,m_k}) \leq z$, a sufficient condition for \eqref{cost-comparison-gbm2-b} is
\begin{equation} \label{cost-comparison-gbm2-b-2}
[k_1 + k_2(z^\eta - y^\eta)]m_k \leq k_1 + k_2(X_k^\eta(\tau_k) - X_k^\eta(\tau_k-)).
\end{equation}
Recall for this second situation, $X_k(\tau_k) > \frac{z}{y}\, X_k(\tau_k-)$ so 
$$X_k^\eta(\tau_k) - X_k^\eta(\tau_k-) > X_k^\eta(\tau_k) - X_k^\eta(\tau_k)\, \mbox{$\left(\frac{y}{z}\right)^\eta$} = \left(1-\mbox{$\left(\frac{y}{z}\right)^\eta$}\right) X_k(\tau_k)$$
and hence
\begin{eqnarray*}
k_1 + k_2(X_k^\eta(\tau_k) - X_k^\eta(\tau_k-)) &\geq& k_1 + k_2\left(1-\mbox{$\left(\frac{y}{z}\right)^\eta$}\right) X_k(\tau_k) \\
&>& k_2\left(1-\mbox{$\left(\frac{y}{z}\right)^\eta$}\right) X_k(\tau_k).
\end{eqnarray*}
Recall the definition of $\widehat{\ell}$ in \eqref{l-r-fns}.  A sufficient condition for \eqref{cost-comparison-gbm2-b-2} and hence for \eqref{cost-comparison-gbm2-b} is
$$[k_1 + k_2(z^\eta - y^\eta)] m_k \leq k_2\left(1-\mbox{$\left(\frac{y}{z}\right)^\eta$}\right) X_k(\tau_k); \mbox{ that is, } \widehat{\ell}(m_k) \leq X_k^\eta(\tau_k).$$
For case (b) as for case (c), the $(\theta_k,Z_k)$ ordering policy is such that $m_k \geq \widehat{m}$ and hence this condition is satisfied.  Hence the total cost up to time $t$ generated by the $(\theta_k,Z_k)$ ordering policy is no greater than that for the $(\tau_k,Y_k)$ ordering policy.
\end{proof}

\bibliographystyle{apalike}

\def\cprime{$'$}

\end{document}